\documentclass[12pt,a4paper]{article}
\usepackage{pgf,tikz,pgfplots}
\pgfplotsset{compat=1.15}
\usetikzlibrary{arrows}

\usepackage[left=2.45cm, top=2.45cm,bottom=2.45cm,right=2.45cm]{geometry}
\usepackage{comment}
\usepackage{mathtools,amssymb,amsthm,mathrsfs,color}
\usepackage[british]{babel}
\usepackage{amsfonts}              % for blackboard bold, etc
\usepackage{xcolor}
\usepackage[
    pdftex,
    colorlinks=true,
    citecolor=blue,
    hyperindex,
    plainpages=false,
    pagebackref=true,
    bookmarksopen,
    bookmarksnumbered
]{hyperref}

\numberwithin{equation}{section}
\numberwithin{figure}{section}

\newtheorem{theorem}{Theorem}[section]
\newtheorem{prop}[theorem]{Proposition}
\newtheorem{lemma}[theorem]{Lemma}
\newtheorem{corollary}[theorem]{Corollary}

\newtheorem{claim}{Claim}[section]

{\theoremstyle{definition}

}
\newcommand{\nc}{\newcommand}
\nc{\grad}{{\mbox{grad}\,}}
\nc{\R}{{\mathbb R}}
\nc{\Rn}{{\mathbb R}^n}
\nc{\N}{{\mathbb N}}
\nc{\Z}{{\mathbb Z}}
\nc{\K}{{\cal K}}
\nc{\kpo}{(\K^d_{2,gp})_o}
\nc{\krep}{\K^d_{3,gp}}
\nc{\BP}{\mathbb{P}}
\nc{\BQ}{\mathbb{Q}}
\nc{\BE}{\mathbb{E}}
\nc{\cH}{\mathcal{H}}
\nc{\cL}{\mathcal{L}}
\nc{\cE}{\mathcal{E}}
\nc{\BS}{\mathbb{S}}
\nc{\BSn}{\mathbb{S}^{n-1}}
\nc{\LL}{L^\circ}
\nc{\bB}{B}
\nc{\sph}{\mathbb{S}^{n-1}}
\nc{\Q}{\mathbb{Q}}
\nc{\cK}{\mathcal{K}}
\nc{\vaps}{\varepsilon}
\nc{\cV}{\mathcal{V}}
\nc{\cP}{\mathcal{P}}
\nc{\cR}{\mathcal{R}}
\nc{\cS}{\mathcal{S}}
\nc{\cB}{\mathcal{B}}
\nc{\cU}{\mathcal{U}}
\nc{\fB}{\mathfrak{B}}
\nc{\cC}{\mathcal{C}}
\nc{\ver}{{\rm{vert\, }}}
\nc{\iF}{\,\raisebox{0.0pt}{$\square$}\,}
\nc{\ind}{\mathbf{1}}	
\nc{\fed}{\llcorner}

\DeclareMathOperator{\Oh}{O}

\usepackage{times}

\begin{document}

\title{Extremizers and stability of the Betke--Weil inequality}
\author{Ferenc A. Bartha\footnotemark[1], Ferenc Bencs\footnotemark[2], K\'aroly J. B\"or\"oczky\footnotemark[3], Daniel Hug\footnotemark[4]}

\renewcommand{\thefootnote}{\fnsymbol{footnote}}

\footnotetext[1]{University of Szeged,  Dugonics ter 13, H-6720 Szeged, Hungary.
E-mail: barfer@math.u-szeged.hu}

\footnotetext[2]{Alfr\'ed R\'enyi Institute of Mathematics, Re\'altanoda u.~13-15, H-1053 Budapest, Hungary. \newline E-mail: ferenc.bencs@gmail.com}

\footnotetext[3]{Alfr\'ed R\'enyi Institute of Mathematics, Re\'altanoda u.~13-15, H-1053 Budapest, Hungary. \newline
Central European University,
Nador utca 9, H-1051 Budapest, Hungary.
 \newline
E-mail: boroczky.karoly.j@renyi.hu}

\footnotetext[4]{Karlsruhe Institute of Technology (KIT), D-76128 Karlsruhe, Germany. E-mail: daniel.hug@kit.edu}

\date{\small \bf To the memory of Ulrich Betke and Wolfgang Weil}

\maketitle
\begin{abstract}
Let $K$ be a compact convex domain in the Euclidean plane. The mixed area $A(K,-K)$ of $K$ and $-K$ can be bounded from above by $1/(6\sqrt{3})L(K)^2$, where $L(K)$ is the perimeter of $K$. This was proved by  Ulrich Betke and Wolfgang Weil (1991). They also showed that if $K$ is a polygon, then equality holds if and
only if $K$ is a regular triangle. We prove that among all convex domains, equality holds only in this case, as conjectured by Betke and Weil. This is achieved by establishing a stronger stability result for  the geometric inequality $6\sqrt{3}A(K,-K)\le L(K)^2$.
\end{abstract}

\section{Introduction}

For  convex domains $K,M$ (compact convex  sets with non-empty interior) in $\R^2$, let $L(K)$ be the perimeter of $K$, let $A(K)$ be the area of $K$, and let
$A(K,M)$ denote the mixed area of $K$ and $M$ (see R.~Schneider \cite{Sch14} or Section~\ref{secmixed}).  In \cite{BeW91}, U.~Betke and W.~Weil proved the following theorem.

\begin{theorem}[Betke, Weil (1991)]
\label{Betke-Weil0}
If $K, M \subset\R^2$ are convex domains, then
\begin{equation}\label{BW1}
 L(K)L(M) \ge 8\, A(K, M)
\end{equation}
with equality if and only if $K$ and $M$ are orthogonal (possibly degenerate) segments.
\end{theorem}

This result has been generalized to higher dimensions in \cite{BH2020}, where also
various improvements in the sense of stability results have been obtained. What makes the variational  analysis of \eqref{BW1} convenient is the fact that $K$ and $M$ can be varied independently of each other and the dependence on $K$ and $M$ is Minkowski linear (in the Euclidean plane). In \cite{BeW91}, U.~Betke and W.~Weil also considered the case where $M=-K$ and found the following sharp geometric inequality.

\begin{theorem}[Betke, Weil (1991)]
\label{Betke-Weil}
If $K$ is a convex domain in $\R^2$, then
\begin{equation}
\label{Betke-Weil-eq}
L(K)^2\geq 6\sqrt{3}\,A(K,-K).
\end{equation}
In addition, if $K$ is a polygon, then equality holds if and only if $K$ is a regular triangle.
\end{theorem}

It is clear from the continuity of the involved functionals that it is sufficient to establish this inequality for convex polygons to
deduce it for general convex domains in the plane. However,
it has remained an open problem to characterize the equality case in (\ref{Betke-Weil-eq}) among all convex domains. We resolve this problem by proving more generally a stability version of Theorem~\ref{Betke-Weil}. We refer to \cite{BH2020} (and in particular to the literature cited there) for a brief introduction to stability improvements of
geometric inequalities.

\begin{theorem}
\label{Betke-Weil-stab}
If $K$ is a convex domain in $\R^2$ and
$$
L(K)^2\leq \left(1+  \varepsilon \right)6\sqrt{3}\,A(K,-K)
$$
for some $\varepsilon\in[0, 2^{-28}]$, then there exists a regular triangle $T$ with centroid $z$ such that
$$
T-z\subset K-z\subset \left(1+400\sqrt{\varepsilon}\right)(T-z).
$$
\end{theorem}

The optimality of the stability exponent $\frac{1}{2}$ of $\varepsilon$ can be seen by considering a regular triangle $T$ of edge length 2. Then we add over each edge $E_i$ of $T$ an isosceles triangle
with height $\sqrt{\varepsilon}$ which has the side $E_i$ in common with $T$ (for $i=1,2,3$). For the resulting
hexagon $H$ we have $L(H)^2-6\sqrt{3}A(H,-H)=36\varepsilon$. However, if $d_{\rm tr}(H)$ is the minimal number $\rho\ge 0$ for which there is a regular triangle $T_0$ with centroid $z$ such that $T_0-z\subset H-z\subset (1+\rho)(T_0-z)$, then it is easy to check that $d_{\rm tr}(H)\ge \sqrt{\varepsilon}$.

\begin{corollary}\label{cor1}
Equality holds in (\ref{Betke-Weil-eq})
if and only if $K$ is a regular triangle.
\end{corollary}

In \cite{BeW91}, Betke and Weil also discuss an application of their Theorem \ref{Betke-Weil}
to an inequality for characteristics of a planar Boolean model. As a consequence of Corollary \ref{cor1},
the equality condition for the lower bound provided in \cite[Theorem 3]{BeW91} now turns into an `if and only if' statement.

For the proof of the inequality \eqref{Betke-Weil-eq}, it is sufficient to consider (convex) polygons with at most $k$ vertices, for any fixed $k\ge 3$.
This task was accomplished by Betke and Weil and we add the observation (extracted from an adaptation of their argument) that for a polygon $P$
which is not a regular $k$-gon  with an odd number of sides, there exist polygons $P'$ with at most $k$ vertices arbitrarily close to $P$
such that
$$
\frac{L(P')^2}{A(P',-P')}<\frac{L(P)^2}{A(P,-P)},
$$
see Proposition \ref{Betke-Weil-deform}. While this can be used to prove the inequality, it does not give control over the equality cases.
To determine all extremal sets, we show that if $K$ is a convex domain which is not too far from a regular triangle, then the
inequality \eqref{Betke-Weil-eq} can be strengthened to a stability result, that is, we show that if $K$  also satisfies
$$
L(K)^2\le (1+\varepsilon)6\sqrt{3}A(K,-K),
$$
then $K$ is $\varepsilon$ close to a regular triangle (if $\varepsilon>0$ is small enough). This local stability result is stated and proved in Section \ref{seclocal} (see Proposition \ref{Betke-Weil-local-stab}). An outline of the
proof of Propositon \ref{Betke-Weil-local-stab}, which is divided into six steps, is given at the beginning of the proof. A major geometric idea, underlying the argument is to approximate $K$ from inside by a triangle $T\subset K$ with maximal area. With  $T$ and $K$ we associate hexagons $H_1,H_2$ with $H_2\supset K$ and $T\subset H_1\subset K$. Another hexagon $H_0$ is derived from $H_1$ so that the perimeter is minimized. Then we show that
$$
L(K)^2-6\sqrt{3}A(K,-K)\ge L(H_0)^2-6\sqrt{3}A(H_2,-H_2)\ge 0.
$$
The fact that the right side is nonnegative is far from obvious. 
More generally, we use a variational argument and validated numerics to establish a lower bound which involves five parameters which determine the shapes of $T$ and $H_2$ (see Lemma \ref{Bekte-Weil-H0H2}). In the course of the proofs, we have to determine various
mixed areas of polygons. These mixed areas are obtained by a classical formula due to Minkowski and by a more recent one which is due to Betke \cite{Bet92} and was first applied in \cite{BeW91}.

\section{Notation and mixed area}
\label{secmixed}

For $p_1,\ldots,p_\ell\in \R^2$, $\ell\in\N$, we write $[p_1,\ldots,p_\ell]$ to denote the convex
hull of the point set $\{p_1,\ldots,p_\ell\}$. In particular, $[p_1,p_2]$ is the segment connecting $p_1$ and $p_2$, and if  $p_1,p_2,p_3$ are not collinear, then
  $[p_1,p_2,p_3]$ is the triangle with vertices $p_1,p_2,p_3$. In addition, the positive hull of $p_1,p_2\in\R^2$ is given by ${\rm pos}\{p_1,p_2\}=\{\alpha_1 p_1+\alpha_2 p_2:\alpha_1,\alpha_2\geq 0  \}$.
The scalar product of $x,y\in\R^2$ is denoted by $\langle x,y\rangle$, and the corresponding Euclidean norm of
$x$ is $\|x\|=\langle x,x\rangle^{1/2}$. In addition, the determinant of the $2\times 2$ matrix with columns $x,y\in\R^2$ is
denoted by $\det(x,y)$. The space of compact convex sets in $\R^2$ is equipped with the Hausdorff metric.

In the following, by a polygon we always mean a convex set.
For a (convex) polygon $P$ in $\R^2$, let $\cU(P)$ denote the finite set of exterior unit normals to the sides of $P$. For
$u\in \cU(P)$, we write $S_P(u)$ to denote the length of the side of $P$ with exterior normal $u$.  As usual, the support function $h_K$ of a compact convex set $K\subset\R^2$ is defined by $h_K(x)=h(K,x)=\max\{\langle x,y\rangle:y\in K\}$ for $ x\in\R^2$.

We recall and will use repeatedly two formulas which allow us to calculate and analyze the mixed area $A(P,Q)$ of two
polygons $P$ and $Q$.  The first is due to Minkowski (see  \cite{Sch14,HW2020}) and states that
\begin{equation}
\label{MinkowskiPQ}
A(P,Q)=\frac12\sum_{u\in \cU(P)} h_Q(u) S_P(u).
\end{equation}
Since $h_{-P}(u)=h_P(-u)$ for $u\in\R^2$, (\ref{MinkowskiPQ}) implies that
\begin{equation}
\label{MinkowskiP-P}
A(P,-P)=\frac12\sum_{u\in \cU(P)} h_P(-u) S_P(u).
\end{equation}
Since $A(P,-P)=A(-P,P)$ (see also below) and $S_{-P}(u)=S_P(-u)$, we also have
$$
A(P,-P)=\frac12\sum_{u\in \cU(P)} h_P(u) S_P(-u).
$$
For instance, if $P$ is a triangle it follows from \eqref{MinkowskiP-P} that $A(P,-P)=2A(P)$.

Another useful formula was established much later by Betke \cite{Bet92}. If $w$ is a unit vector
 with $w\not\in \cU(P)\cup \cU(-Q)$, then
\begin{equation}
\label{BetkePQ}
2\cdot A(P,Q)=\sum_{\substack{u\in \cU(P),v\in\cU(Q)\\ w\in{\rm pos}\{u,-v\}}} |\det(u,v)| S_P(u)S_Q(v).
\end{equation}
In particular, (\ref{BetkePQ}) yields that if $w\not\in \cU(P)$ for a (fixed) unit vector $w$, then
\begin{equation}
\label{BetkeP-P}
A(P,-P)=\sum_{\substack{\{u,v\}\in \cU(P)\\ w\in{\rm pos}\{u,v\}}} |\det(u,v)|S_P(u)S_P(v),
\end{equation}
where the summation extends over all  subsets of $\cU(P)$ of cardinality two for which $w\in{\rm pos}\{u,v\}$ holds (the factor $2$ from the preceding formula cancels, since we do not consider ordered pairs).
It is formula (\ref{BetkeP-P}) that was used in a clever way by Betke and Weil \cite{BeW91} to prove
the Betke--Weil inequality stated in Theorem~\ref{Betke-Weil}.

Minkowski proved that the notion of mixed area can be extended to any pair $K,M$ of compact convex sets in $\R^2$
(see \cite{Sch14,HW2020}), and for compact convex sets $K_1,K_2,M\subset\R^2$ and $\alpha_1,\alpha_2\geq 0$,
it is known that
\begin{align*}
\nonumber
A(K,M)&=A(M,K),\\
A(K+z_1,M+z_2)&=A(K,M)\mbox{ \ for $z_1,z_2\in\R^2$,}
\\
\nonumber
A(\Phi K,\Phi M)&=|\det \Phi|\cdot A(K,M) \mbox{ \ for $\Phi\in {\rm GL}(2,\R)$,}\\
\nonumber
A(K,K)&=A(K),\mbox{ \ which is the area of }K,\\
%\label{linearity}
A(\alpha_1K_1+\alpha_2K_2,M)&= \alpha_1A(K_1,M)+\alpha_2A(K_2,M),\\
%\label{monotonicity}
A(K_1,M)&\leq A(K_2,M)\mbox{ \ if }K_1\subset K_2.
\end{align*}

We note that it is a subtle issue to decide under which conditions on convex polygons $P\subset Q$  the inequality
$A(P,-P)\leq A(Q,-Q)$  is strict. For example, let $P$ be a  triangle with its centroid at the origin $o$. Let
$v_1,v_2,v_3$ denote the vertices of $P$, and let $Q$ be the hexagon with vertices
$v_1,v_2,v_3,-v_1,-v_2,-v_3$. Then $P\subset Q$ and $P\neq Q$, actually $A(Q)=2A(P)$, and still we have
$A(P,-P)=2A(P)=A(Q)=A(Q,-Q)$.

Finally, we recall that $A(\cdot,\cdot)$ is additive (a valuation) in both arguments. By this we mean that if $K,M,L$ are compact convex sets in the plane and $K\cup M$ is also convex, then
$$
A(K\cup M,L)+A(K\cap M,L)=A(K,L)+A(M,L).
$$
By symmetry the same property holds for the second argument.

\section{An auxiliary result for associated hexagons}
\label{secpolytopes}

Let $T\subset K$ denote a triangle of maximal area contained in $K$. Let $v_1,v_2,v_3$ be the vertices of $T$,
let $a_i$ be the side opposite to $v_i$, whose length is also denoted by $a_i$ for $i=1,2,3$, and let $h_i$ be the height of $T$ corresponding to $a_i$. Then we have
$$
2A(T)=a_1h_1=a_2h_2=a_3h_3.
$$
3We observe that the line passing through $v_i$ and parallel to the side $a_i$ is a supporting line to $K$,  by the maximality of the area of $T$. The width of $K$ orthogonal to $a_i$ can be expressed in the form $(1+t_i)h_i$ for some $t_i\in [ 0,1]$, where $t_i\le 1$ follows since $T$ has maximal area among all triangles in $K$. (The width of $K$ orthogonal to $a_i$ equals the length of the projection of $K$ to a line orthogonal to $a_i$.) It follows that
$K$ is contained in a circumscribed hexagon $H_2$ such that, for $i=1,2,3$, $H_2$ has two  sides parallel to $a_i$, one of which contains $v_i$ and has  length $(t_j+t_k)a_i$, $\{i,j,k\}=\{1,2,3\}$, and the
opposite side has length $(1-t_i)a_i$ (see Figure \ref{Fig1}). These assertions follow by elementary geometry from
the similarity of corresponding triangles. In fact, with the notation from Figure \ref{Fig1} we have
$$
{\|w_{32}'-w_{32}\|}=\frac{a_2}{a_3}\cdot {\|w_{32}-v_3\|} \quad \text{and}\quad {\|w_{32}'-w_{32}\|}=\frac{a_2}{h_2}\cdot {t_2h_2},
$$
hence
$$
\|w_{32}'-w_{32}\|=t_2a_2,\quad \|w_{32}-v_3\|=t_2a_3,
$$
and similarly for permutations of the indices. Moreover,
$$
\frac{\|w_{32}'-w_{12}'\|}{a_2}=\frac{h_2(1+t_2)}{h_2},
$$
so that $\|w_{32}'-w_{12}'\|=(1+t_2)a_2$, and hence
$$
\|w_{32}-w_{12}\|=(1+t_2)a_2-\|w_{32}'-w_{32}\|-\|w_{12}'-w_{12}\|=(1-t_2)a_2.
$$

Next, for $i=1,2,3$ and $\{i,j,k\}=\{1,2,3\}$ we choose a point $q_i\in [w_{ji},w_{ki}]\cap K$. Then we
 define the (possibly degenerate) hexagon $H_1=[v_1,q_2,v_3,q_1,v_2,q_3]\subset K$.
In addition, let $p_i$ be the point on the line determined by $[w_{ji},w_{ki}]$  (and parallel to $a_i$) which lies on the perpendicular bisector of the side $a_i$ of $T$, and let $H_0$ be the hexagon with vertices $v_1,p_2,v_3,p_1,v_2,p_3$  (see Figure \ref{Fig1}). Note that in general $H_0$ may not be convex, but the restricted choice of parameters encountered in the following will always ensure that $H_0$ is convex.

\begin{center}
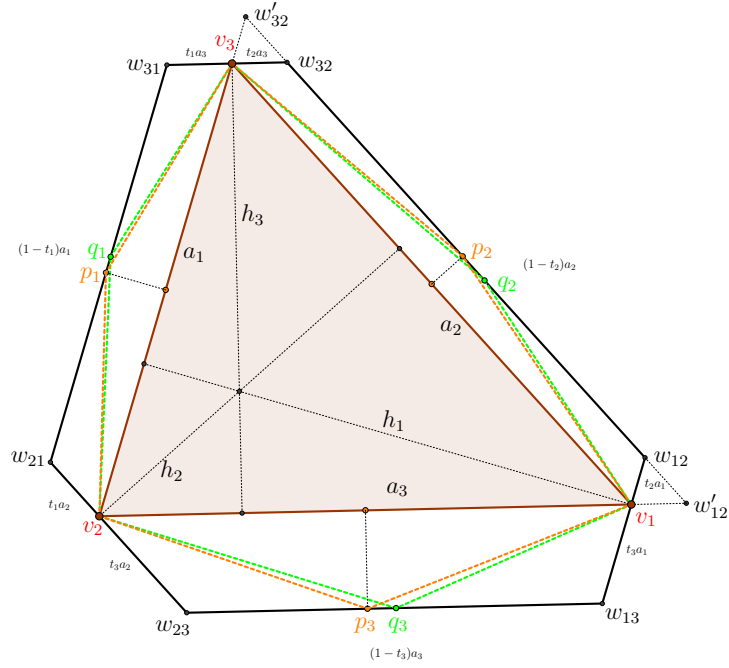
\begin{figure}

\centering
\scalebox{0.4}{
\definecolor{xdxdff}{rgb}{0.49019607843137253,0.49019607843137253,1.}
\definecolor{uuuuuu}{rgb}{0.26666666666666666,0.26666666666666666,0.26666666666666666}
\definecolor{zzttqq}{rgb}{0.6,0.2,0.}
\definecolor{ududff}{rgb}{0.30196078431372547,0.30196078431372547,1.}
\definecolor{qqwuqq}{rgb}{0.,0.39215686274509803,0.}
\begin{tikzpicture}[line cap=round,line join=round,>=triangle 45,x=1.0cm,y=1.0cm]
%\clip(2.,-10.) rectangle (10.,12.);
\fill[line width=2.pt,color=zzttqq,fill=zzttqq,fill opacity=0.10000000149011612] (4.6705026639999945,7.326140110000009) -- (17.804855917999998,-7.289238140000015) -- (0.30537635999999335,-7.6789815600000155) -- cycle;
\draw [line width=2.pt,color=zzttqq] (4.6705026639999945,7.326140110000009)-- (17.804855917999998,-7.289238140000015);
\draw [line width=2.pt,color=zzttqq] (17.804855917999998,-7.289238140000015)-- (0.30537635999999335,-7.6789815600000155);
\draw [line width=2.pt,color=zzttqq] (0.30537635999999335,-7.6789815600000155)-- (4.6705026639999945,7.326140110000009);
\draw [line width=2.pt] (6.479191485415802,7.3664227118132795)-- (18.256021058308615,-5.738357970189127);
\draw [line width=2.pt] (2.5264865420783305,7.278389194144292)-- (-1.3038295043376478,-5.888322215410651);
\draw [line width=2.pt] (3.179723019591773,-10.87743852542114)-- (16.849583437660886,-10.572987291165703);
\draw [line width=2.pt,dash pattern=on 3pt off 3pt,color=green] (4.6705026639999945,7.326140110000009)-- (12.97621239524618,0.13680003766967452);
\draw [line width=2.pt,dash pattern=on 3pt off 3pt,color=green] (12.97621239524618,0.13680003766967452)-- (17.804855917999998,-7.289238140000015);
\draw [line width=2.pt,dash pattern=on 3pt off 3pt,color=green] (17.804855917999998,-7.289238140000015)-- (10.067657827557744,-10.724032404976462);
\draw [line width=2.pt,dash pattern=on 3pt off 3pt,color=green] (10.067657827557744,-10.724032404976462)-- (0.30537635999999335,-7.6789815600000155);
\draw [line width=2.pt,dash pattern=on 3pt off 3pt,color=green] (0.30537635999999335,-7.6789815600000155)-- (0.6769335190250809,0.9205506773987393);
\draw [line width=2.pt,dash pattern=on 3pt off 3pt,color=green] (0.6769335190250809,0.9205506773987393)-- (4.6705026639999945,7.326140110000009);
\draw [line width=1.pt,dotted] (1.7753065570277258,-2.626096507717178)-- (17.804855917999998,-7.289238140000015);
\draw [line width=1.pt,dotted] (5.002362613469259,-7.5743716211699645)-- (4.6705026639999945,7.326140110000009);
\draw [line width=1.pt,dotted] (10.18018905245636,1.1951834462874285)-- (0.30537635999999335,-7.6789815600000155);
\draw [line width=1.pt,dotted] (4.6705026639999945,7.326140110000009)-- (5.121667804308615,8.877020279810893);
%\draw [line width=1.pt,dotted] (4.6705026639999945,7.326140110000009)-- (3.0612967996623537,9.116799454589374);
\draw [line width=2.pt,] (2.5264865420783305,7.278389194144292)-- (6.479191485415802,7.3664227118132795);
\draw [line width=2.pt] (18.256021058308615,-5.738357970189127)-- (16.849583437660886,-10.572987291165703);
\draw [line width=2.pt] (-1.3038295043376478,-5.888322215410651)-- (3.179723019591773,-10.87743852542114);
%\draw [line width=1.pt,dotted] (3.512461939970974,10.66767962440026)
%-- (18.256021058308615,-5.738357970189127);
%\draw [line width=1.pt,dotted] (3.512461939970974,10.66767962440026)-- (2.5264865420783305,7.278389194144292);
\draw (2.916657273999994,1.3630657839999971) node[anchor=north west] {\scalebox{1.9}{$a_1$}};
\draw (11.374089487999996,-1.0533434200000074) node[anchor=north west] {\scalebox{1.9}{$a_2$}};
\draw (9.601782915999995,-6.4) node[anchor=north west] {\scalebox{1.9}{$a_3$}};
\draw (-0.23,1.5) node[anchor=north west] {\scalebox{1.9}{\begin{color}{green}$q_1$\end{color}}};
\draw (-0.47,0.7) node[anchor=north west] {\scalebox{1.9}{\begin{color}{orange}$p_1$\end{color}}};
\draw (-2.508025840000069,1.50256295979999983) node[anchor=north west] {$(1-t_1)a_1$};
\draw (13.25,0.413) node[anchor=north west] {\scalebox{1.9}{\begin{color}{green}$q_2$\end{color}}};
\draw (12.4,1.5) node[anchor=north west] {\scalebox{1.9}{\begin{color}{orange}$p_2$\end{color}}};
\draw (14.100729613999996,1.0138348619999977) node[anchor=north west] {$(1-t_2)a_2$};
\draw (9.659218439999995,-10.880005894000026) node[anchor=north west] {\scalebox{1.9}{\begin{color}{green}$q_3$\end{color}}};
\draw (8.559218439999995,-10.880005894000026) node[anchor=north west] {\scalebox{1.9}{\begin{color}{orange}$p_3$\end{color}}};
\draw (9.059218439999995,-11.880005894000026) node[anchor=north west] {$(1-t_3)a_3$};
\draw (9.464346729999995,-4.028295980000116) node[anchor=north west] {\scalebox{1.9}{$h_1$}};
\draw (2.16,-5.68) node[anchor=north west] {\scalebox{1.9}{$h_2$}};
\draw (4.826400031999994,2.8051164379999998) node[anchor=north west] {\scalebox{1.9}{$h_3$}};
\draw (-1.5,-7.0) node[anchor=north west] {$t_1a_2$};
\draw (0.5,-9.0) node[anchor=north west] {$t_3a_2$};
\draw (-0.4,-7.7) node[anchor=north west] {\scalebox{1.9}{\begin{color}{red}$v_2$\end{color}}};
\draw (-2.75,-5.4) node[anchor=north west] {\scalebox{1.9}{$w_{21}$}};
\draw (2.00,-11.0) node[anchor=north west] {\scalebox{1.9}{$w_{23}$}};

\draw (17.8,-7.4) node[anchor=north west] {\scalebox{1.9}{\begin{color}{red}$v_1$\end{color}}};
\draw (16.8,-10.6) node[anchor=north west] {\scalebox{1.9}{$w_{13}$}};
\draw (18.4,-5.4) node[anchor=north west] {\scalebox{1.9}{$w_{12}$}};
\draw (19.7,-6.8) node[anchor=north west] {\scalebox{1.9}{$w_{12}'$}};

\draw (17.5,-8.5) node[anchor=north west] {$t_3a_1$};
\draw (18.1,-6.3) node[anchor=north west] {$t_2a_1$};
\draw (3.95,8.4) node[anchor=north west] {\scalebox{1.9}{\begin{color}{red}$v_3$\end{color}}};
\draw (6.7,7.7) node[anchor=north west] {\scalebox{1.9}{$w_{32}$}};
\draw (5.25,9.5) node[anchor=north west] {\scalebox{1.9}{$w_{32}'$}};
\draw (1.05,7.6) node[anchor=north west] {\scalebox{1.9}{$w_{31}$}};
\draw (3.0,8.0) node[anchor=north west] {$t_1a_3$};
\draw (5.0,8.0) node[anchor=north west] {$t_2a_3$};
\begin{scriptsize}
%\draw [fill=zzttqq] (4.6705026639999945,7.326140110000009) circle (3.5pt);
%\draw [fill=zzttqq] (17.804855917999998,-7.289238140000015) circle (3.5pt);
%\draw [fill=zzttqq] (0.30537635999999335,-7.6789815600000155) circle (3.5pt);
\draw [fill=uuuuuu] (6.479191485415802,7.3664227118132795) circle (2.0pt);%w_{32}
\draw [fill=uuuuuu] (2.5264865420783305,7.278389194144292) circle (2.0pt);%w_{31}
\draw [fill=uuuuuu] (-1.3038295043376478,-5.888322215410651) circle (2.0pt);%w_{21}
\draw [fill=uuuuuu] (3.179723019591773,-10.87743852542114) circle (2.0pt);%w_{23}
\draw [fill=uuuuuu] (16.849583437660886,-10.572987291165703) circle (2.0pt);%w_{13}
\draw [fill=uuuuuu] (18.256021058308615,-5.738357970189127) circle (2.0pt);%w_{12}

\draw [fill=uuuuuu] (5.121667804308615,8.877020279810893) circle (2.0pt);%w_{32}'
\draw [fill=uuuuuu] (19.61354476, -7.248955540) circle (2.0pt);%w_{12}'

\draw [fill=uuuuuu] (1.7753065570277258,-2.626096507717178) circle (2.0pt);%base point of h_1
\draw [fill=uuuuuu] (10.18018905245636,1.1951834462874285) circle (2.0pt);%base point of h_2
\draw [fill=uuuuuu] (5.002362613469258,-7.574371621169964) circle (2.0pt);%base point of h_3
\draw [fill=uuuuuu] (4.912481939192984,-3.538729346165252) circle (2.0pt);%intersection of hights

\draw [fill=green] (0.6769335190250809,0.9205506773987393) circle (2.5pt);%q_1 previous color was xdxdff
\draw [fill=green] (12.97621239524618,0.13680003766967452) circle (2.5pt);%q_2
\draw [fill=green] (10.067657827557744,-10.724032404976462) circle (2.5pt);%q_3

\draw [fill=orange] (0.5240175483120926,0.3949020280728398) circle (2.5pt);%p_1
\draw [fill=orange] (2.487939511999994,-0.17642072500000347) circle (2.5pt);%foot of p_1
\draw [fill=orange] (12.25830754861328,0.9356555791751305) circle (2.5pt);%p_2
\draw [fill=orange] (11.237679290999997,0.01845098499999805) circle (2.5pt);%foot of p_2
\draw [fill=orange] (9.127740996788008,-10.744965964681803) circle (2.5pt);%p_3
\draw [fill=orange] (9.06,-7.48) circle (2.5pt);%foot of p_3

\draw [line width=1.pt,dotted] (11.237679290999997,0.01845098499999805)
-- (12.25830754861328,0.9356555791751305);

\draw [line width=1.pt,dotted] (9.06,-7.48)
-- (9.127740996788008,-10.744965964681803);

\draw [line width=1.pt,dotted] (2.487939511999994,-0.17642072500000347)
-- (0.5240175483120926,0.3949020280728398);

\draw [line width=2.pt,dashed,color=orange] (17.804855917999998,-7.289238140000015)-- (12.25830754861328,0.9356555791751305);

\draw [line width=1.pt,dotted] (17.804855917999998,-7.289238140000015)-- (19.61354476, -7.248955540);

\draw [line width=1.pt,dotted]
(6.479191485415802,7.3664227118132795)-- (5.121667804308615,8.877020279810893);

\draw [line width=1.pt,dotted] (18.256021058308615,-5.738357970189127)-- (19.61354476, -7.248955540);

\draw [line width=2.pt,dashed,color=orange] (12.25830754861328,0.9356555791751305)-- (4.6705026639999945,7.326140110000009);

\draw [line width=2.pt,dashed,color=orange] (4.6705026639999945,7.326140110000009)-- (0.5240175483120926,0.3949020280728398);

\draw [line width=2.pt,dashed,color=orange] (0.5240175483120926,0.3949020280728398)--
(0.31,-7.68);

\draw [line width=2.pt,dashed,color=orange] (0.31,-7.68)--
(9.127740996788008,-10.744965964681803);

\draw [line width=2.pt,dashed,color=orange] (9.127740996788008,-10.744965964681803)--
(17.8,-7.29);

\draw [fill=zzttqq] (4.6705026639999945,7.326140110000009) circle (3.5pt);%v_3
\draw [fill=zzttqq] (17.804855917999998,-7.289238140000015) circle (3.5pt);%v_1
\draw [fill=zzttqq] (0.30537635999999335,-7.6789815600000155) circle (3.5pt);%v_2

\end{scriptsize}
\end{tikzpicture}}
\caption{Illustration
of the geometric construction with the triangle \begin{color}{red} {$T=[v_1,v_2,v_3]$}\end{color}, the outer hexagon $H_2=[w_{32},w_{31},w_{21},w_{23},w_{13},w_{12}]$, the inscribed hexagon \begin{color}{green} {$H_1=[v_1,q_2,v_3,q_1,v_2,q_3]$}\end{color} and the hexagon \begin{color}{orange} {$H_0=[v_1,p_2,v_3,p_1,v_2,p_3]$}\end{color}. In particular, we have $\|w_{ij}-v_i\|=t_ja_i$ and $\|w_{ij}-w_{kj}\|=(1-t_j)a_j$ for $\{i,j,k\}=\{1,2,3\}$. %Thanks $\heartsuit$ to Annette Hug
}
\label{Fig1}
\end{figure}
\end{center}

As $H_1\subset K\subset H_2$ and $L(H_0)\leq L(H_1)$, we  have
\begin{align}
\label{KH}
& {L(K)^2}-6\sqrt{3}\,{A(K,-K)}\\
&\qquad\qquad\geq  {L(H_1)^2}-6\sqrt{3}\,{A(H_2,-H_2)}
 \geq  {L(H_0)^2}-6\sqrt{3}\,{A(H_2,-H_2)}.\nonumber
\end{align}

Clearly, $T$ is also a triangle of maximal area contained in $H_1$.
As among convex domains of given area, the maximal area of an inscribed triangle is the smallest for ellipses (see Blaschke \cite{Bla17}, Sas \cite{Sas39} and Schneider \cite[Theorem~10.3.3]{Sch14}), we have
$$
A(T)\geq \frac{3\sqrt{3}}{4\pi} \cdot A(H_1)>0.4 \cdot A(H_1)=0.4(1+t_1+t_2+t_3)A(T),
$$
and hence
\begin{equation}
\label{t1+t2+t3}
t_1+t_2+t_3<1.5.
\end{equation}

The following lemma is the basis for obtaining better bounds on $t_1,t_2,t_3$ if we know that
${L(K)^2}-6\sqrt{3}\,{A(K,-K)}$ is small.

\begin{lemma}
\label{Bekte-Weil-H0H2}
If $a_1=2$,  $a_2, a_3\in [2,2+\frac{1}{6}]$
  and $t_1,t_2,t_3\in[0,\frac{1}{6}]$, then the hexagons $H_0$ and $H_2$ constructed as above satisfy
$$
L(H_0)^2-6\sqrt{3}\, A(H_2,-H_2)\geq
(a_2-2)^2+(a_3-2)^2+(t_1-t_0)^2+(t_2-t_0)^2+(t_3-t_0)^2
$$
for $t_0=(t_1+t_2+t_3)/3$.
\end{lemma}

\noindent {\bf Remark}
In the lemma, we do not need $K$, only the triangle $T$ and $t_1,t_2,t_3\geq 0$ are required to define $H_0$ and $H_2$. Moreover, although $H_0$ will be convex in the situation of the lemma, this will not be needed in the argument.

\medskip

\begin{proof} By the translation invariance of the mixed area, we can assume that $v_2$ is the origin.
Then we obtain from \eqref{MinkowskiP-P} that
\begin{align*}
A(H_2,-H_2)&=\frac{1}{2}\left\{(1-t_2)a_2\cdot 0+(t_1+t_2)a_3t_3h_3+(1-t_1)a_1h_1\right.\\
&\left.\qquad\qquad +(t_1+t_3)a_2h_2(1+t_2)+
(1-t_3)a_3h_3+(t_2+t_3)a_1h_1t_1\right\}\\
&=2A(T)(1+t_1t_2+t_2t_3+t_3t_1).\\
\intertext{By Heron's formula,}
%\label{Ahi}
2A(T)&=a_1h_1=a_2h_2=a_3h_3\\
%\label{AT}
&=\frac12\sqrt{(a_1+a_2+a_3)(-a_1+a_2+a_3)(a_1-a_2+a_3)(a_1+a_2-a_3)},\\
\intertext{and in addition we have}
%\label{LH0}
L(H_0)&=
\sum_{i=1}^3\sqrt{a_i^2+4t_i^2h_i^2}.
\end{align*}
Setting $b_i=2t_ih_i=\frac{4\,A(T)t_i}{a_i}$ for $i=1,2,3$, it follows from the
 Minkowski inequality (or equivalently, the triangle inequality for $(a_i,b_i)$, $i=1,2,3$) that
\begin{align*}
L(H_0)^2&=\left(\sum_{i=1}^3\sqrt{a_i^2+b_i^2}\right)^2\\
&\geq (a_1+a_2+a_3)^2+(b_1+b_2+b_3)^2\\
&=(a_1+a_2+a_3)^2+16\,A(T)^2(t_1/a_1+t_2/a_2+t_3/a_3)^2\\
&=:f_1(a_2,a_3,t_1,t_2,t_3).
\end{align*}
For the subsequent analysis, we set  $f_2(a_2,a_3,t_1,t_2,t_3):=16A(T)^2$, hence
$$
f_2(a_2,a_3,t_1,t_2,t_3)=(a_1+a_2+a_3)(-a_1+a_2+a_3)(a_1-a_2+a_3)(a_1+a_2-a_3)
$$
and
%$$f_1(a_2,a_3,t_1,t_2,t_3)=(a_1+a_2+a_3)^2+16\,A(T)^2(t_1/a_1+t_2/a_2+t_3/a_3)^2,$$
%hence
$$
f_1(a_2,a_3,t_1,t_2,t_3)=(a_1+a_2+a_3)^2+f_2(a_2,a_3,t_1,t_2,t_3)(t_1/a_1+t_2/a_2+t_3/a_3)^2,
$$
and finally we set
$$
f(a_2,a_3,t_1,t_2,t_3):=f_1(a_2,a_3,t_1,t_2,t_3)-3\sqrt{3}\sqrt{f_2(a_2,a_3,t_1,t_2,t_3)}\,(1+t_1t_2+t_2t_3+t_3t_1).
$$
Thus we obtain
\begin{equation}
\label{LAf}
L(H_0)^2-6\sqrt{3}\cdot A(H_2,-H_2)\geq f(a_2,a_3,t_1,t_2,t_3),
\end{equation}
%where, using the functions
%$$f_1(a_2,a_3,t_1,t_2,t_3)=(a_1+a_2+a_3)^2+16\,A(T)^2(t_1/a_1+t_2/a_2+t_3/a_3)^2$$
%and
%$$f_2(a_2,a_3,t_1,t_2,t_3)=16A(T)^2,$$
%we have
%\begin{align*}
%f(a_2,a_3,t_1,t_2,t_3)&=f_1(a_2,a_3,t_1,t_2,t_3)-3\sqrt{3}\sqrt{f_2(a_2,a_3,t_1,t_2,t_3)}(1+t_1t_2+t_2t_3+t_3t_1),\\
%f_2(a_2,a_3,t_1,t_2,t_3)&=(a_1+a_2+a_3)(-a_1+a_2+a_3)(a_1-a_2+a_3)(a_1+a_2-a_3),\\
%f_1(a_2,a_3,t_1,t_2,t_3)&=(a_1+a_2+a_3)^2+f_2(a_2,a_3,t_1,t_2,t_3)(t_1/a_1+t_2/a_2+t_3/a_3)^2.
%\end{align*}
In the following, we consider
\begin{align*}
W&:=\left\{(a_2,a_3,t_1,t_2,t_3)^\top\in\R^5:\,\mbox{$a_2, a_3\in [2,2+\frac16]$
  and $t_1,t_2,t_3\in[0,\frac16]$}\right\},\\
z_t&:=(2,2,t,t,t)^\top,\quad t\in [0,\tfrac{1}{6}],
\end{align*}
and the orthonormal basis
\begin{align*}
e_1&=(1,0,0,0,0)^\top,\quad
e_2=(0,1,0,0,0)^\top,\quad
e_3=\left(0,0,\frac1{\sqrt{2}},\frac{-1}{\sqrt{2}},0\right)^\top,\\
e_4&=\left(0,0,\frac1{\sqrt{6}},\frac1{\sqrt{6}},\frac{-2}{\sqrt{6}}\right)^\top,\quad
e_5=\left(0,0,\frac1{\sqrt{3}},\frac1{\sqrt{3}},\frac1{\sqrt{3}}\right)^\top.
\end{align*}
We write $Df$ for the derivative and $D^2f$ to denote the Hessian of $f$.
Using  a computer algebra system (for convenience) or direct calculations, we obtain that
if $t\in[0,\tfrac{1}{6}]$, then
\begin{equation}
\label{fclose}
f(z_t)=0,\quad Df(z_t)=o.
\end{equation}
By the Taylor formula and (\ref{fclose}) there is a
$\xi\in(0,1)$ such that
$$
f(x)=\frac12\,\left\langle x-z_t,D^2f(z_t+\xi(x-z_t)) (x-z_t)\right\rangle,\quad x\in W.
$$
By relation \eqref{LAf}, Lemma~\ref{Bekte-Weil-H0H2} follows once we have shown that
\begin{equation}
\label{fLemma31}
f(x) \geq \|x-z_{t}\|^2
\end{equation}
for $x=(a_2,a_3,t_1,t_2,t_3)^\top\in W$ and $t=(t_1+t_2+t_3)/3$.

Since for $x=(a_2,a_3,t_1,t_2,t_3)^\top\in W$, we have $x-z_{t}\in e_5^\perp$ (we write $e_5^\perp$ for the orthogonal complement of $e_5$)
with $t=(t_1+t_2+t_3)/3$ and $z_t+\xi(x-z_t)\in W$, the
proof will be finished if we can verify that
\begin{equation}\label{evbound}
\langle v, D^2f(\bar{x}) v\rangle\ge 2\|v\|^2,\quad \bar{x}\in W,v\in e_5^\perp.
\end{equation}

Using a computer algebra system
(such as SageMath or Maple) or tedious calculations, we obtain that
$$D^2f(2,2,0,0,0)=
\left(\begin{array}{rrrrr}
12 & -6 & 0& 0 & 0 \\
-6 & 12 & 0& 0 & 0 \\
0  & 0 & 24 & -12 & -12 \\
0 & 0 & -12 & 24 & -12 \\
0 & 0 & -12 & -12 & 24
\end{array}\right)
$$
has the eigenvalues
$6,18,36,36,0$, and as associated pairwise orthogonal eigenvectors, one can choose $(1,1,0,0,0)^\top$ to correspond to the eigenvalue $6$,
$(1,-1,0,0,0)^\top$ to correspond to $18$, $e_3$ and $e_4$  to correspond to $36$, and finally
$e_5$ to correspond to $0$.

Define the orthogonal matrix $S:=(e_1\ldots e_5)\in \Oh(5)$ and write $y=(y_1,\ldots,y_5)^\top\in\R^5$. Further, we define
$$
\tilde{f}(y_1,\ldots,y_5):=f\left(\sum_{i=1}^5 y_ie_i\right)=f(Sy),
$$
hence $f(x)=\tilde{f}(S^\top x)$, $x\in\R^5$.
By the chain rule, for $x\in W$ and $v\in\R^5$ we have
$$
\langle v, D^2f(x) v\rangle=\langle S^\top v,D^2\tilde{f}(S^\top x)S^\top v\rangle.
$$
In addition, note that $\|S^\top v\|^2=\|v\|^2$ and
\begin{align*}
S^\top(W)&\subset \left[2,2+\tfrac{1}{6}\right]^2\times \left[-\tfrac{1}{6}\sqrt{\tfrac{2}{3}},\tfrac{1}{6}\sqrt{\tfrac{2}{3}}\right]^2\times \left[0,\tfrac{\sqrt{3}}{6}\right]\\
&
\subset \widetilde{W}:=\left[2,2+\tfrac{1}{6}\right]^2\times [-0.14,0.14]^2\times[0,0.3].
\end{align*}
Here we use that for $y=S^\top (a_2,a_3,t_1,t_2,t_3)^\top$ we have
\begin{align*}
y_3^2+y_4^2&=\frac{1}{2}(t_1-t_2)^2+\frac{1}{6}(t_1+t_2-2t_3)^2=\frac{2}{3}\frac{1}{2}\left[(t_3-t_2)^2+(t_3-t_1)^2+(t_2-t_1)^2\right]\\
&\le \frac{2}{3}\max\{t_1,t_2,t_3\}^2=\frac{2}{3}\frac{1}{6^2}
\end{align*}
and
$$
0\le y_5=\frac{1}{\sqrt{3}}(t_1+t_2+t_3)\le \sqrt{3}\cdot\frac{1}{6}.
$$

Moreover, $v\in  e_5^\bot$ if and only if $\langle S^\top v,e_5^\circ\rangle=0$, where $e_5^\circ=(0,0,0,0,1)^\top$.
Hence, \eqref{evbound} follows if we can verify that
$$
\langle \tilde{v},D^2\tilde{f}(y) \tilde{v}\rangle\ge 2\|\tilde{v}\|^2, \quad y\in\widetilde{W},\, \tilde{v}\in (e_5^\circ)^\bot.
$$
Writing $H(y):=\left(D^2\tilde{f}(y)_{ij}\right)_{i,j=1}^{n-1}$ for the $(n-1)\times(n-1)$ matrix (principal minor) obtained from the $n\times n$ Hessian matrix representing  $D^2\tilde{f}(y)$ with respect to the standard basis $e_1^\circ,\ldots,e_5^\circ$ of $\R^5$, we want to verify that
\begin{equation}\label{eeq1}
\langle \bar{v}, H(y) \bar{v}\rangle\ge 2\|\bar{v}\|^2,\quad y\in\widetilde{W},\, \bar{v}\in \R^4,
\end{equation}
that is to say, all eigenvalues of $H(y)$ are at least $2$. Let
$$
\Xi:=\left\{(\bar{v}_1,\ldots,\bar{v}_4)^\top\in [-1,1]^4:\bar{v}_i=1\text{ for some $i\in\{1,\ldots,4\}$}\right\}.
$$
By scaling invariance of \eqref{eeq1} with respect to  $\bar{v}\in \R^4$, \eqref{eeq1} is equivalent to
\begin{equation}\label{eq:extraevbound}
\langle \bar{v}, H(y) \bar{v}\rangle\ge 2\|\bar{v}\|^2,\quad y\in\widetilde{W},\, \bar{v}\in {\color{red}\Xi},
\end{equation}
Since all eigenvalues of $H((2,2,0,0,0)^\top)$ are positive, this holds if and
only if all eigenvalues of $H(y)^2$ are at least $4$, for $y\in\widetilde{W}$. The latter means that we have to show that
$\langle \bar{v}, H(y)^2 \bar{v}\rangle\ge 4\|\bar{v}\|^2$ for $y\in\widetilde{W}$ and $ \bar{v}\in \R^4$
or, equivalently,
\begin{equation}\label{evbound2}
\| H(y) \bar{v}\|^2\ge 4\|\bar{v}\|^2,\quad y\in\widetilde{W}, \, \bar{v}\in \R^4.
\end{equation}
Again by scaling invariance of \eqref{evbound2} with respect to $\tilde{v}$, \eqref{evbound2} is in turn equivalent to
\begin{equation}\label{evbound3}
\| H(y) \bar{v}\|^2\ge 4\|\bar{v}\|^2,\quad y\in\widetilde{W},\, \bar{v}\in \Xi.
\end{equation}

Direct rigorous numerical analysis of the eigendecomposition
of the hessian $D^2\tilde{f}$ may be challenging
due to requiring too many subdivisions of $\widetilde{W}$ in order
to achieve the required precision \cite{rump, hladik, jp, hartman}.
As both, $\widetilde{W}$ and $\Xi$ are compact, finite dimensional
and as the desired inequalities, either \eqref{eq:extraevbound}
or \eqref{evbound3}, are expected to be strict,
they are well suited for being studied by rigorous numerics \cite{moore, alefeld, warwick}.

Namely,
for small $\widetilde{W}' \subset \widetilde{W}$,
and $\Xi' \subset \Xi$, we
preform the following procedure with
all computations being carried out rigorously using
interval arithmetic and automatic differentiation
\cite{moore, alefeld, warwick, griewank}.
First, we bound the jet of $\tilde{f}$ up to Taylor-coefficients
of degree $6$ over $\widetilde{W}'$.
In order to increase precision and eliminate
some of the dependency issues,
for a given $\widetilde{W}'$,
the degree 6 jet of $f$ is bound
both over $\widetilde{W}'$ and over the midpoint of $\widetilde{W}'$.
Hence, using the multivariate Taylor-expansions with the
appropriate remainder term, we obtain enhanced bounds on the
Taylor-coefficients of $\tilde{f}$ over $\widetilde{W}'$ and, in turn,
a better enclosure of the hessian matrix $\mathrm{H}=D^2\tilde{f}$.
Second, we test if we can guarantee that
the inequality \eqref{eq:extraevbound}
or \eqref{evbound3} for all $v \in \Xi'$ holds.
If that is not the case,
then an adaptive bisection scheme of $\widetilde{W}' \times \Xi'$
is utilized and the arising subsets of $\widetilde{W}' \times \Xi'$
are processed separately.

We have implemented our software using the package CAPD \cite{capd} and
verified both inequalities independently and successfully.
The required number of
subsets (of $\widetilde{W} \times \Xi$) and the associated computational
times (without parallelization on an i7-9750) were
\begin{itemize}
    \item \eqref{eq:extraevbound}: $25880$ subsets, 8m14s;
    \item \eqref{evbound3}: $2440$ subsets, 46s.
\end{itemize}
We note that the increased complexity of \eqref{eq:extraevbound} is
most likely just an artefact of the naive computation of the inner product
and could be decreased (to that of \eqref{evbound3}) by choosing a more
efficient evaluation scheme.
The source code and output logs are available at \cite{fb-code}.

In particular, both \eqref{eq:extraevbound}, \eqref{evbound3}, and, in turn, Lemma~\ref{Bekte-Weil-H0H2} have been verified.
\end{proof}

The following two claims will be used in the proof of Proposition~\ref{Betke-Weil-local-stab}.

\begin{claim}
\label{TinT0}
If the  regular triangle $T_0$ of side length $b$ contains a triangle $T$ which has a side of length at most $a$, where
$\frac{b}2\leq a\leq b$, then $A(T)\leq \frac{a}{b}\,A(T_0)$.
\end{claim}

\begin{proof} Let $\ell$ be the line containing a side of $T$ of length at most $a$.
We may assume that the vertex $v$ of $T$ opposite to $\ell\cap T$ is also a vertex of $T_0$.

If the distance of $\ell$ from $v$ (that is, the height of $T$) is at most $b\sqrt{3}/2$, then we are done.
Therefore we may now assume that the distance of $\ell$ from $v$ is larger than $b\sqrt{3}/2$.
Let $D_v$ be the circular disc with center $v$ and radius $b\sqrt{3}/2$. Since the side $\ell\cap T$ is disjoint from $D_v$, it lies in one of the two connected components of $T_0\setminus D_v$. It follows that $T$ is contained in one of the two triangles which are obtained by cutting $T_0$ into two sub-triangles by the height emanating from $v$, thus
$A(T)\leq \frac{1}{2}\,A(T_0)$.
\end{proof}

\begin{claim}
\label{triangleineq}
If $\varrho_1,\varrho_2,\varrho_3$ are the side lengths of a triangle $P$  and $A(P)\geq \xi\,  \varrho_1$,  for some $\xi\in[0,\varrho_1]$, then
$$
\varrho_2+\varrho_3\geq \varrho_1+\frac{\xi^2}{\varrho_1}.
$$
\end{claim}

\begin{proof} The height $h$ of $P$ corresponding to $\varrho_1$ is at least $2\xi$, and $\varrho_2+\varrho_3$ is
minimized under this condition if $\varrho_2=\varrho_3$, thus
$$
\varrho_2+\varrho_3\geq 2\sqrt{\left(\frac{\varrho_1}2\right)^2+h^2}\geq
\sqrt{\varrho_1^2+16\xi^2}=\varrho_1\sqrt{1+\frac{16\xi^2}{\varrho_1^2}}\geq \varrho_1+\frac{\xi^2}{\varrho_1},
$$
which proves the claim.
\end{proof}

\section{Local stability}
\label{seclocal}

For any convex domain $K$, let
$d_{\rm tr}(K)$ be the minimal $\rho\geq 0$ such that
there exists a regular triangle $T$ with centroid $z$ satisfying
$$
T-z\subset K-z\subset (1+\rho)(T-z).
$$
In particular, $d_{\rm tr}(K)$ measures how close $K$ is to a suitable regular triangle.

\begin{prop}
\label{Betke-Weil-local-stab}
Suppose that  $K$ is a convex domain with $d_{\rm tr}(K)\leq 6^{-2}$ and
\begin{equation}\label{locref1}
L(K)^2\leq (1+\varepsilon)6\sqrt{3}\,A(K,-K)
\end{equation}
for some $\varepsilon\in[0,(6\cdot 180)^{-2}]$. Then $d_{\rm tr}(K)\leq 400\sqrt{\varepsilon}$.
\end{prop}

\begin{proof}
Let $d_{\rm tr}(K)=\eta\leq 6^{-2}$ and $\varepsilon$ be as in the statement of the proposition.
There exists a regular triangle $T_0$ of side length $b$ containing $K$ such that
a translate of $\frac1{1+\eta}\,T_0$ is contained in $K$. For a triangle $T\subset K$  of maximal area contained in $K$,
we have
\begin{equation}
\label{TT0eta}
A(T)\geq \frac{A(T_0)}{(1+\eta)^2}>\frac{A(T_0)}{1+3\eta}.
\end{equation}

From now on, we use the notions and auxiliary constructions introduced for $K$ and $T$
at the beginning of Section~\ref{secpolytopes}, including
the hexagons $H_0,H_1,H_2$, the parameters $t_1,t_2,t_3\geq 0$, etc.

The main part of the proof is divided into several steps. In Step~1, we prove that
$T$ is $\sqrt{\varepsilon}$ close to a regular triangle,
Step~2 shows that $A(H_2,-H_2)-A(K,-K)$ is $\varepsilon$ small.
Based on these findings, Step~3 verifies that
 if $H_2$ is close
to $T$ in the sense that
$\max\{t_1,t_2,t_3\}\leq 100\,\sqrt{\varepsilon}$
  (see (\ref{tismall})), then Proposition~\ref{Betke-Weil-local-stab} holds.

The rest of the argument is indirect. Starting from Step~4, we assume that the assumption
$\max\{t_1,t_2,t_3\}>100\,\sqrt{\varepsilon}$
 (see (\ref{tibig0})) is satisfied and derive a contradiction.
Under this assumption, we prove in Step~4 that $H_1$ is reasonably close to $H_0$ in the
sense that
$\|p_i-q_i\|$ is reasonably small for $i=1,2,3$ (see (\ref{piqi})).
Then Step~5 verifies that $K\subset D=\frac12\,H_1+\frac12\,H_2$
and clearly $D\subset H_2$.
Finally, in Step~6, we prove that the gap between
$A(H_2,-H_2)$ and $A(D,-D)$ (and hence the gap between
$A(H_2,-H_2)$ and $A(K,-K)$ by Step~5) is too large, which yields the desired contradiction.

\bigskip

\noindent {\bf Step 1} {\it $T$ is $\sqrt{\varepsilon}$ close to a regular triangle.}

\medskip

By scaling invariance of the assertion (and symmetry), we may assume that the side lengths of $T$ satisfy
$$
a_1\leq a_2\leq a_3 \mbox{ \ and \ }a_1=2.
$$
Assuming $a_1<\frac{b}{2}$, we can apply Claim~\ref{TinT0} with $a=b/2$ and get
$$
A(T)\le \frac{1}{2}A(T_0)\le \frac{1}{2}   (1+\eta)^2 A(T)
$$
by (\ref{TT0eta}), which is a contradiction, since $\eta\le 6^{-2}$. Hence, we have $\frac{b}{2}\le a_1\le b$,
and another application of Claim~\ref{TinT0} now yields that
$$
A(T)\le \frac{a_1}{b}A(T_0)\le \frac{a_1}{b}(1+3\eta)A(T),
$$
so that
$$\frac{b}{1+3\eta}\leq a_1=2\leq a_3\leq b,$$
and hence
\begin{equation}
\label{airatio}
 a_3\leq b\le 2(1+3\eta)= 2+6\eta\leq 2+\frac{1}{6}.
\end{equation}
It also follows from (\ref{TT0eta}) that
$$
(1+t_1+t_2+t_3)A(T)=A(H_1)\leq A(K)\leq A(T_0)<(1+3\eta)A(T),
$$
therefore
\begin{equation}
\label{titau}
 t_i<\frac{1}{12}<\frac{1}{6} \mbox{ \ for $i=1,2,3$.}
\end{equation}
Thus, the condition $d_{\rm tr}(K)\le 6^{-2}$ ensures that Lemma~\ref{Bekte-Weil-H0H2} can be applied.

Hence, by a combination of \eqref{KH} with Lemma~\ref{Bekte-Weil-H0H2} and the assumption \eqref{locref1} of the proposition, we see that
\begin{align}
\label{aiclose20}
(a_i-2)^2&\leq 6\sqrt{3}\,A(K,-K)\varepsilon< 11\,A(K,-K)\varepsilon\mbox{ \ for $i=2,3$},\\
\label{ticloset0}
(t_i-t)^2&\leq 6\sqrt{3}\,A(K,-K)\varepsilon< 11\,A(K,-K)\varepsilon \mbox{ \ for }t=\tfrac{1}{3}(t_1+t_2+t_3), i=1,2,3.
\end{align}
To estimate $A(K,-K)$, we observe that the height $h_1$ of $T$ corresponding to the side $a_1=2$
satisfies $h_1\leq \sqrt{a_3^2-1}<2$, thus $A(T)\leq \sqrt{a_3^2-1}<2$ by (\ref{airatio}), and hence
\begin{equation}
\label{AK-K}
A(K,-K)\leq A(T_0,-T_0)=2A(T_0)\leq 2(1+3\eta)A(T)<\frac{39}{9}<5.
\end{equation}
Now (\ref{aiclose20}) and (\ref{ticloset0}) imply that
\begin{align}
\label{aiclose2}
2\leq a_i&\leq 2+7\sqrt{\varepsilon}\mbox{ \ \ \ for $i=2,3$},\\
\label{ticloset}
|t_i-t|&\leq 7\sqrt{\varepsilon}\mbox{ \ \ \ for $i=1,2,3$}.
\end{align}
We observe that $\frac{a_1^2+a_2^2-a_3^2}{a_2}$ is an increasing function of $a_2\geq 2$ as $a_3\geq a_1$.
Writing $\alpha_i$ to denote the angle of $T$ opposite to the side  $a_i$ and using that
$7^2\sqrt{\varepsilon}\le 1$, we have
$$
\cos\alpha_3=\frac{a_1^2+a_2^2-a_3^2}{2a_1a_2}\geq \frac{8- (2+7\sqrt{\varepsilon})^2}{8}\geq
\frac{4-29\sqrt{\varepsilon}}{8}\geq \frac12-4\sqrt{\varepsilon}>0.
$$
In particular, we have $\alpha_1\le \alpha_2\le\alpha_3<\pi/2$, $\alpha_3\ge \pi/3$ and $\alpha_1\le \pi/3$.

Since $\cos'( s)=-\sin s\leq -\sqrt{3}/2$ if $s\in [\frac{\pi}3,\frac{\pi}2]$, the mean value theorem then  implies that
$$
\alpha_3\leq \frac{\pi}3+\frac2{\sqrt{3}}\, 4\sqrt{\varepsilon}\leq
\frac{\pi}3+5\sqrt{\varepsilon}.
$$
Since $\pi\le\alpha_1+2\alpha_3\le \alpha_1+2(\frac{\pi}{3}+5\sqrt{\varepsilon})$, we conclude that
\begin{equation}
\label{alhai}
\frac{\pi}3-10\sqrt{\varepsilon}\leq \alpha_1\leq \alpha_2\leq\alpha_3\leq \frac{\pi}3+5\sqrt{\varepsilon}.
\end{equation}
As a first estimate, we deduce $\tan\alpha_3\leq 2$, thus
$\tan'( s)=1+(\tan s)^2\leq 5$ if $s\in [\frac{\pi}3,\alpha_3]$. In particular,
(\ref{alhai}) yields
\begin{equation}
\label{tanalha}
\sqrt{3}-50\sqrt{\varepsilon}\leq \tan\alpha_1\le \tan\alpha_2\leq
\tan\alpha_3\leq \sqrt{3}+25\sqrt{\varepsilon}.
\end{equation}

Now let $T'_1$ be the regular triangle of edge length $2$ positioned in such a way that the side $a_1$ is common with $T$ and ${\rm int}\,T\cap{\rm int}\,T'_1\neq \emptyset$. Recalling that $v_2$ and $v_3$ are the endpoints of $a_1$,
we write $v'_1$ to denote the third vertex of $T'_1$, $z_1$ to denote the centroid of $T'_1$ and
$m=\frac12(v_2+v_3)$ to denote the midpoint of $a_1$.
As $a_3\geq a_2$ and $\frac{\pi}{2}>\alpha_3\geq \frac{\pi}3$, there exists a point $q$ such that $v'_1\in[m,q]$ and
 $v_1\in [v_3,q]$. In addition, we consider the intersection point $\{p\}=[q,m]\cap[v_1,v_2]\neq \emptyset$ (see Figure \ref{Fia2}).

\begin{center}
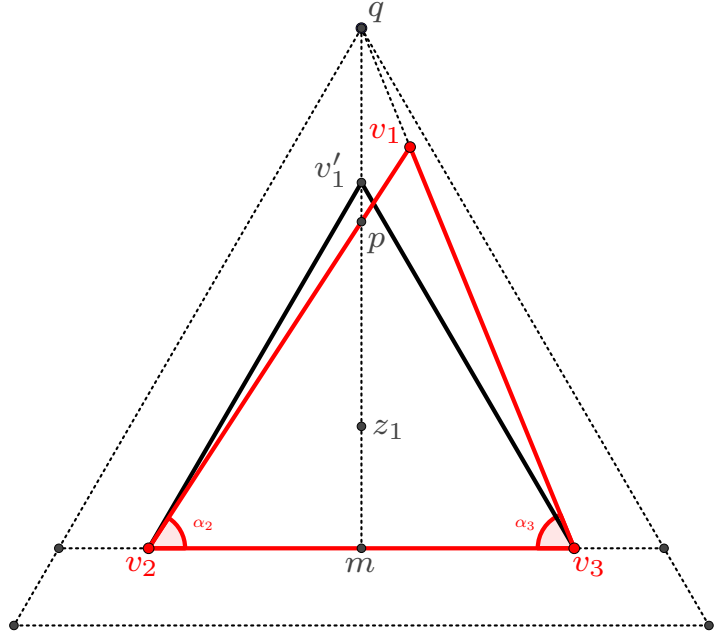
\begin{figure}

\centering
\scalebox{0.8}{

\definecolor{xdxdff}{rgb}{0.49019607843137253,0.49019607843137253,1.}
%\definecolor{qqwuqq}{rgb}{0.,0.39215686274509803,0.}
\definecolor{qqwuqq}{rgb}{1.,0.,0.}
\definecolor{uuuuuu}{rgb}{0.26666666666666666,0.26666666666666666,0.26666666666666666}
\definecolor{ffqqqq}{rgb}{1.,0.,0.}
\definecolor{ududff}{rgb}{0.30196078431372547,0.30196078431372547,1.}
\definecolor{cqcqcq}{rgb}{0.7529411764705882,0.7529411764705882,0.7529411764705882}
\begin{tikzpicture}[line cap=round,line join=round,>=triangle 45,x=1.0cm,y=1.0cm]
%\draw [color=cqcqcq,, xstep=1.0cm,ystep=1.0cm] (-1.44,-4.39) grid (25.2,13.31);
%\clip(-1.44,-4.39) rectangle (25.2,13.31);
\draw [shift={(6.,2.)},line width=2.pt,color=qqwuqq,fill=qqwuqq,fill opacity=0.10000000149011612] (0,0) -- (0.:0.6) arc (0.:57.11262292598501:0.6) -- cycle;
\draw [shift={(13.,2.)},line width=2.pt,color=qqwuqq,fill=qqwuqq,fill opacity=0.10000000149011612] (0,0) -- (112.09789211568872:0.6) arc (112.09789211568872:180.:0.6) -- cycle;
\draw [line width=2.pt,color=ffqqqq] (6.,2.)-- (13.,2.);
\draw [line width=2.pt] (13.,2.)-- (9.5,8.062177826491071);
\draw [line width=2.pt] (9.5,8.062177826491071)-- (6.,2.);
\draw [line width=1.pt,dotted] (9.5,2.)-- (9.5,8.062177826491071);
\draw [line width=2.pt,color=ffqqqq] (6.,2.)-- (10.3,8.65);
\draw [line width=1.pt,dotted] (9.5,10.620370370370376)-- (14.47697315384761,2.);
\draw [line width=1.pt,dotted] (14.47697315384761,2.)-- (13.,2.);
\draw [line width=1.pt,dotted] (6.,2.)-- (4.523026846152386,2.);
\draw [line width=1.pt,dotted] (4.523026846152386,2.)-- (9.5,10.620370370370376);
\draw [line width=1.pt,dotted] (10.3,8.65)-- (9.5,10.620370370370376);
\draw [line width=2.pt,color=ffqqqq] (13.,2.)-- (10.3,8.65);
\draw [line width=1.pt,dotted] (9.5,8.062177826491071)-- (9.5,10.620370370370376);
\draw [line width=1.pt,dotted] (4.523026846152386,2.)-- (3.784540269228577,0.7209037280603428);
\draw [line width=1.pt,dotted] (3.784540269228577,0.7209037280603428)-- (15.215459730771412,0.7209037280603469);
\draw [line width=1.pt,dotted] (15.215459730771412,0.7209037280603469)-- (14.47697315384761,2.);
\begin{scriptsize}
\draw [fill=ffqqqq] (6.,2.) circle (2.5pt);
\draw[color=ffqqqq] (5.89,1.7) node {\scalebox{1.9}{$v_2$}};
\draw [fill=ffqqqq] (13.,2.) circle (2.5pt);
\draw[color=ffqqqq] (13.25,1.7) node {\scalebox{1.9}{$v_3$}};
\draw [fill=uuuuuu] (9.5,8.062177826491071) circle (2.0pt);
\draw[color=uuuuuu] (9.0,8.27) node {\scalebox{1.9}{$v_1'$}};
\draw [fill=uuuuuu] (9.5,2.) circle (2.0pt);
\draw[color=uuuuuu] (9.48,1.72) node {\scalebox{1.9}{$m$}};
\draw [fill=uuuuuu] (9.5,4.020725942163691) circle (2.0pt);
\draw[color=uuuuuu] (9.95,3.97) node {\scalebox{1.9}{$z_1$}};
\draw [fill=ffqqqq] (10.3,8.65) circle (2.5pt);
\draw[color=ffqqqq] (9.9,8.87) node {\scalebox{1.9}{$v_1$}};
\draw [fill=uuuuuu] (9.5,10.620370370370376) circle (2.0pt);
\draw[color=uuuuuu] (9.74,10.88) node {\scalebox{1.9}{$q$}};
\draw [fill=uuuuuu] (9.5,7.412790697674417) circle (2.0pt);
\draw[color=uuuuuu] (9.75,7.05) node {\scalebox{1.9}{$p$}};
\draw[color=qqwuqq] (6.9,2.39) node {$\alpha_2$};
\draw[color=qqwuqq] (12.2,2.39) node {$\alpha_3$};
\draw [fill=uuuuuu] (4.523026846152386,2.) circle (2.0pt);
\draw [fill=uuuuuu] (14.47697315384761,2.) circle (2.0pt);
\draw [fill=uuuuuu] (9.5,10.620370370370374) circle (2.0pt);
\draw [fill=xdxdff] (9.5,10.620370370370376) circle (2.5pt);
\draw [fill=uuuuuu] (9.5,10.620370370370374) circle (2.0pt);
\draw [fill=uuuuuu] (3.784540269228577,0.7209037280603428) circle (2.0pt);
\draw [fill=uuuuuu] (15.215459730771412,0.7209037280603469) circle (2.0pt);
\end{scriptsize}
\end{tikzpicture}}
\caption{Illustration of the triangles $T=[v_1,v_2,v_3]$, $T_1'=[v_1',v_2,v_3]$ together with the auxiliary points $m,p,q,z_1$ in the case where $p\in [v_1',m]$ . In the case   $p\notin [v_1',m]$ (not shown in the figure) we have  $T_1'\subset T$ and $p=q=v_1$.
}
\label{Fia2}
\end{figure}
\end{center}

 We deduce from (\ref{tanalha}) that
 $$\|q-v'_1\|=\tan(\alpha_3)-\sqrt{3}\leq  25\sqrt{\varepsilon}$$
 and
 $$
 \|p-m\|=\tan(\alpha_2)\ge \tan(\alpha_1)\ge \sqrt{3}-50\sqrt{\varepsilon},
 $$
 and if $p\in [v_1',m]$, then also
 $$
 \|p-v'_1\|=\|v_1'-m\|-\|p-m\|=\sqrt{3}-\tan(\alpha_2)\leq 50\sqrt{\varepsilon}.
 $$
 Therefore, $\|v'_1-z_1\|=\frac2{\sqrt{3}}$ implies that in any case
\begin{equation}\label{T1primeT}
\left(1-25\sqrt{3}\,\sqrt{\varepsilon}\right)(T'_1-z_1)\subset T-z_1\subset
\left(1+\frac{25\sqrt{3}}2\,\sqrt{\varepsilon}\right)(T'_1-z_1).
\end{equation}
For the regular triangle
$$
T_1=z_1+\left(1-25\sqrt{3}\,\sqrt{\varepsilon}\right)(T'_1-z_1),
$$
with centroid at $z_1$, we have
\begin{equation}
\label{T1T}
T_1-z_1\subset T-z_1\subset \left(1+70\sqrt{\varepsilon}\right)(T_1-z_1).
\end{equation}
In addition, let $z$ be the centroid of $T$. The argument above also shows that
$$
v_1\in v'_1+\frac{25\sqrt{3}}{2}\,\sqrt{\varepsilon}\,(T'_1-z_1).
$$
Since $z=\frac13(v_1+v_2+v_3)$ and $z_1=\frac13(v'_1+v_2+v_3)$, we  get
\begin{equation}
\label{T1z}
\begin{array}{rcl}
z-z_1&\in& \frac{25\sqrt{3}}{2\cdot 3}\,\sqrt{\varepsilon}\cdot (T'_1-z_1)\subset
10\sqrt{\varepsilon}\cdot (T_1-z_1),\\[1ex]
z_1-z&\in& 20\sqrt{\varepsilon}\cdot (T_1-z_1),
\end{array}
\end{equation}
where for the second containment we used that $-(T_1-z_1)\subset 2(T_1-z_1)$. In summary,
\eqref{T1T} and \eqref{T1z} show that the regular triangle $T_1$ is a very good approximation of $T$.

Note that $p_i,q_i$ lie on the same side of $H_2$ (above $a_i$) for $i=1,2,3$. This follows from
$\alpha_i\le \frac{\pi}{2}$,
$a_i\ge 2$, \eqref{airatio} and \eqref{titau}, which imply that $t_ia_j\le \frac{1}{12}(2+\frac{1}{6})<0.2<1\le \frac{a_k}{2}$ for $i,j,k\in \{1,2,3\}$. Moreover, for $i=1,2,3$ we have
\begin{equation}\label{bhnew}
1.77>(1+4 \sqrt{\varepsilon})(\sqrt{3}+25\sqrt{\varepsilon})\ge \frac{a_i}{2}\tan(\alpha_3)\ge h_i\ge \frac{a_i}{2}\tan(\alpha_1)\ge \sqrt{3}-50\sqrt{\varepsilon}>1.68,
\end{equation}
and hence in particular
\begin{equation}
\label{ATTprime1}
2>1.77>A(T)=h_1 >1.68>1.5.
\end{equation}

\bigskip

\noindent {\bf Step 2} {\it $A(K,-K)$ and $A(H_2,-H_2)$ are $\varepsilon$ close, and $L(H_0)$, $L(H_1)$, $L(K)$ are $\varepsilon$ close.}

\medskip

We deduce from (\ref{AK-K}) that $6\sqrt{3}\,A(K,-K)\le 26\sqrt{3}<50$. Therefore \eqref{locref1},
Lemma~\ref{Bekte-Weil-H0H2},
 $H_1\subset K\subset H_2$ and $L(H_0)\leq L(H_1)$ imply that
\begin{align}
\nonumber
50\varepsilon&\geq L(K)^2-6\sqrt{3}\,A(K,-K)
\geq L(K)^2-6\sqrt{3}\,A(H_2,-H_2)\\
\label{L2-A}
&\geq  L(H_1)^2-6\sqrt{3}\,A(H_2,-H_2)\geq L(H_0)^2-6\sqrt{3}\,A(H_2,-H_2)\geq 0.
\end{align}
In particular,
\begin{align*}
L(K)^2-6\sqrt{3}\,A(K,-K)&\le 50\varepsilon,\\
6\sqrt{3}\,A(H_2,-H_2)-L(K)^2&\le 0,
\end{align*}
which yields
\begin{equation}
\label{AKH2}
A(H_2,-H_2)\leq A(K,-K)+5\varepsilon.
\end{equation}
Using $L(K)\geq L(H_1)\geq L(H_0)\geq L(T)\geq 6$ and
\begin{align*}
L(K)^2-6\sqrt{3}\,A(H_2,-H_2)&\le 50\varepsilon,\\
6\sqrt{3}\,A(H_2,-H_2)-L(H_1)^2&\le 0,
\end{align*}
we get
$$
50\varepsilon\geq L(K)^2-L(H_1)^2=(L(K)-L(H_1))(L(K)+L(H_1))\geq 12(L(K)-L(H_1)).
$$
We deduce that
\begin{equation}
\label{LKH1}
L(K)\leq L(H_1)+5\varepsilon.
\end{equation}
A similar argument shows that
\begin{equation}
\label{LH1H0}
L(H_1)\leq L(H_0)+5\varepsilon.
\end{equation}

This completes Step 2.

\bigskip

For the remaining part of the proof, we set $\gamma=10^2$. In the following, we distinguish whether
\begin{equation}
\label{tismall}
\max\{t_1,t_2,t_3\}\leq \gamma\sqrt{\varepsilon}
\end{equation}
is satisfied or not. If \eqref{tismall} holds, then $H_2$ is $\sqrt{\varepsilon}$ close to $K$ (see the argument below).

\bigskip

\noindent {\bf Step 3} {\it If \eqref{tismall} holds, then $d_{\rm tr}(K)\le 4\gamma \sqrt{\varepsilon}$.}

\medskip

It follows from (\ref{T1T}) and $T\subset K$ that
\begin{equation}
\label{T1inK}
T_1-z_1\subset K-z_1.
\end{equation}
Using \eqref{tismall} and recalling that  $z$ is the centroid of $T$, we get
$$
K-z\subset H_2-z\subset \left(1+3 \max\{t_1,t_2,t_3\}\right)(T-z)\subset (1+3\gamma\sqrt{\varepsilon})(T-z).
$$
Therefore (\ref{T1T}) and (\ref{T1z}) imply that if (\ref{tismall}) holds, then
\begin{align}
\nonumber
K-z_1&= K-z+(z-z_1) \subset
(1+3\gamma\sqrt{\varepsilon})(T-z_1)+3\gamma\sqrt{\varepsilon}(z_1-z)\\
&\subset [1+70\sqrt{\varepsilon}+3\gamma \sqrt{\varepsilon}+210\gamma\varepsilon +60 \gamma\varepsilon]
(T_1-z_1)\nonumber \\
\label{tismallK}
&\subset \left(1+4\gamma\sqrt{\varepsilon}\right)(T_1-z_1).
\end{align}
In view of (\ref{T1inK}), we conclude Proposition~\ref{Betke-Weil-local-stab} if  (\ref{tismall}) holds.

\bigskip

It remains to consider the case where
\begin{equation}
\label{tibig0}
\max\{t_1,t_2,t_3\}> \gamma\sqrt{\varepsilon}.
\end{equation}
This case will finally lead to a contradiction.

It follows from (\ref{ticloset}) and (\ref{tibig0}) that
\begin{equation}
\label{tibig}
\min\{t_1,t_2,t_3\}\geq 86\sqrt{\varepsilon}.
\end{equation}

\noindent {\bf Step 4} {\it Assuming (\ref{tibig0}), $H_1$ is reasonably close to $H_0$.}

\medskip

More precisely, we claim that
\begin{equation}
\label{piqi}
\|p_i-q_i\|< 0.06<0.1\mbox{ \ \ for $i=1,2,3$},
\end{equation}
which is what we mean by saying that $H_1$ is ``reasonably close" to $H_0$.

Let $\{i,j,k\}=\{1,2,3\}$, and assume that $q_i\neq p_i$ (otherwise (\ref{piqi}) readily holds). For the line
 $\ell_i$ through $p_i$ and $q_i$ and parallel to $a_i$, let $\tilde{v}_j$ be the reflection of $ {v}_j$
through $\ell_i$, and hence $p_i$ is the midpoint of $[\tilde{v}_j,v_k]$. For the triangle
$\widetilde{T}_i=[q_i,\tilde{v}_j,v_k]$, we have
$$
A(\widetilde{T}_i)=2A([q_i,p_i,v_k])=\|p_i-q_i\|t_ih_i.
$$
Recall from \eqref{bhnew} that $1.68< h_i <1.77$. Using \eqref{aiclose2} and $t_i\leq\frac{1}{12}$ by  (\ref{titau}), we get
$$
\|\tilde{v}_j-v_k\|=2\|p_i-v_k\|\leq 2\left(\frac{a_i}2+\frac{1}{12} h_i\right)\leq a_i+\frac{1}{6} h_i
\le 2+\frac{7}{6\cdot 180}+\frac{1}{6}\cdot 1.77\le 2.31.
$$
Therefore
$$
A(\widetilde{T}_i)\geq \|p_i-q_i\|t_i\cdot 1.68\ge \frac{\|p_i-q_i\|t_i}{2.31}\cdot 1.68\cdot \|\tilde{v}_j-v_k\|.
$$
Since
$$
\frac{\|p_i-q_i\|t_i}{2.31}\cdot 1.68\le \frac{a_i}{12\cdot 2.31}\cdot 1.68 <a_i\le \|\tilde{v}_j-v_k\|,
$$
Claim~\ref{triangleineq} can be applied. In Combination with
$t_i\geq 86\sqrt{\varepsilon}$ (by (\ref{tibig})), this  leads to
\begin{align*}
\|q_i-v_j\| +\|q_i-v_k\|&=\|q_i-\tilde{v}_j\| +\|q_i-v_k\|\\
&\geq
\|\tilde{v}_j-v_k\|+\frac{1}{2.31}\left(\frac{\|p_i-q_i\|t_i\cdot 1.68}{2.31}\right)^2\\
&= \|p_i-v_j\| +\|p_i-v_k\| + \|p_i-q_i\|^2\,t_i^2\,\frac{1.68^2}{2.31^3}\\
&\geq \|p_i-v_j\| +\|p_i-v_k\| + \|p_i-q_i\|^2 \,\frac{86^2\cdot 1.68^2}{2.31^3}\, \varepsilon   .
\end{align*}
We deduce from (\ref{LH1H0}) that
$$5\,\varepsilon\geq \|p_i-q_i\|^2 \,\frac{86^2\cdot 1.68^2}{2.31^3}\, \varepsilon,$$
and hence (\ref{piqi}) follows.

\medskip

\noindent {\bf Step 5} {\it Assuming (\ref{tibig0}), we have $K\subset D:=\frac12\,H_1+\frac12\,H_2$.}

\medskip

The polygon $D\subset H_2$
has twelve sides (see Figure \ref{Fig3}).

\begin{center}
\begin{figure}

\centering
\scalebox{0.4}{
\definecolor{xdxdff}{rgb}{0.49019607843137253,0.49019607843137253,1.}
\definecolor{uuuuuu}{rgb}{0.26666666666666666,0.26666666666666666,0.26666666666666666}
\definecolor{zzttqq}{rgb}{0.6,0.2,0.}
\definecolor{ududff}{rgb}{0.30196078431372547,0.30196078431372547,1.}
\definecolor{qqwuqq}{rgb}{0.,0.39215686274509803,0.}
\begin{tikzpicture}[line cap=round,line join=round,>=triangle 45,x=1.0cm,y=1.0cm]
%\clip(2.,-10.) rectangle (10.,12.);
\fill[line width=2.pt,color=zzttqq,fill=zzttqq,fill opacity=0.10000000149011612] (4.6705026639999945,7.326140110000009) -- (17.804855917999998,-7.289238140000015) -- (0.30537635999999335,-7.6789815600000155) -- cycle;
\draw [line width=2.pt,color=zzttqq] (4.6705026639999945,7.326140110000009)-- (17.804855917999998,-7.289238140000015);
\draw [line width=2.pt,color=zzttqq] (17.804855917999998,-7.289238140000015)-- (0.30537635999999335,-7.6789815600000155);
\draw [line width=2.pt,color=zzttqq] (0.30537635999999335,-7.6789815600000155)-- (4.6705026639999945,7.326140110000009);
\draw [line width=2.pt] (6.479191485415802,7.3664227118132795)-- (18.256021058308615,-5.738357970189127);
\draw [line width=2.pt] (2.5264865420783305,7.278389194144292)-- (-1.3038295043376478,-5.888322215410651);
\draw [line width=2.pt] (3.179723019591773,-10.87743852542114)-- (16.849583437660886,-10.572987291165703);
\draw [line width=2.pt,dash pattern=on 3pt off 3pt,color=green] (4.6705026639999945,7.326140110000009)-- (12.97621239524618,0.13680003766967452);
\draw [line width=2.pt,dash pattern=on 3pt off 3pt,color=green] (12.97621239524618,0.13680003766967452)-- (17.804855917999998,-7.289238140000015);
\draw [line width=2.pt,dash pattern=on 3pt off 3pt,color=green] (17.804855917999998,-7.289238140000015)-- (10.067657827557744,-10.724032404976462);
\draw [line width=2.pt,dash pattern=on 3pt off 3pt,color=green] (10.067657827557744,-10.724032404976462)-- (0.30537635999999335,-7.6789815600000155);
\draw [line width=2.pt,dash pattern=on 3pt off 3pt,color=green] (0.30537635999999335,-7.6789815600000155)-- (0.6769335190250809,0.9205506773987393);
\draw [line width=2.pt,dash pattern=on 3pt off 3pt,color=green] (0.6769335190250809,0.9205506773987393)-- (4.6705026639999945,7.326140110000009);
\draw [line width=1.pt,dotted] (1.7753065570277258,-2.626096507717178)-- (17.804855917999998,-7.289238140000015);
\draw [line width=1.pt,dotted] (5.002362613469259,-7.5743716211699645)-- (4.6705026639999945,7.326140110000009);
\draw [line width=1.pt,dotted] (10.18018905245636,1.1951834462874285)-- (0.30537635999999335,-7.6789815600000155);

\draw [line width=2.pt,] (2.5264865420783305,7.278389194144292)-- (6.479191485415802,7.3664227118132795);
\draw [line width=2.pt] (18.256021058308615,-5.738357970189127)-- (16.849583437660886,-10.572987291165703);
\draw [line width=2.pt] (-1.3038295043376478,-5.888322215410651)-- (3.179723019591773,-10.87743852542114);

\draw (2.916657273999994,1.3630657839999971) node[anchor=north west] {\scalebox{1.9}{$a_1$}};
\draw (11.374089487999996,-1.0533434200000074) node[anchor=north west] {\scalebox{1.9}{$a_2$}};
\draw (9.601782915999995,-6.4) node[anchor=north west] {\scalebox{1.9}{$a_3$}};
\draw (-0.23,1.5) node[anchor=north west] {\scalebox{1.9}{\begin{color}{green}$q_1$\end{color}}};
\draw (-0.47,0.7) node[anchor=north west] {\scalebox{1.9}{\begin{color}{orange}$p_1$\end{color}}};

\draw (13.25,0.413) node[anchor=north west] {\scalebox{1.9}{\begin{color}{green}$q_2$\end{color}}};
\draw (12.4,1.5) node[anchor=north west] {\scalebox{1.9}{\begin{color}{orange}$p_2$\end{color}}};

\draw (9.659218439999995,-10.880005894000026) node[anchor=north west] {\scalebox{1.9}{\begin{color}{green}$q_3$\end{color}}};
\draw (8.559218439999995,-10.880005894000026) node[anchor=north west] {\scalebox{1.9}{\begin{color}{orange}$p_3$\end{color}}};

\draw (-0.4,-7.7) node[anchor=north west] {\scalebox{1.9}{\begin{color}{red}$v_2$\end{color}}};
\draw (-2.75,-5.4) node[anchor=north west] {\scalebox{1.9}{$w_{21}$}};
\draw (2.00,-11.0) node[anchor=north west] {\scalebox{1.9}{$w_{23}$}};

\draw (17.8,-7.4) node[anchor=north west] {\scalebox{1.9}{\begin{color}{red}$v_1$\end{color}}};
\draw (16.8,-10.6) node[anchor=north west] {\scalebox{1.9}{$w_{13}$}};
\draw (18.4,-5.4) node[anchor=north west] {\scalebox{1.9}{$w_{12}$}};

\draw (3.95,8.4) node[anchor=north west] {\scalebox{1.9}{\begin{color}{red}$v_3$\end{color}}};
\draw (6.7,7.7) node[anchor=north west] {\scalebox{1.9}{$w_{32}$}};

\draw (1.05,7.6) node[anchor=north west] {\scalebox{1.9}{$w_{31}$}};

\begin{scriptsize}
\draw [fill=uuuuuu] (6.479191485415802,7.3664227118132795) circle (2.0pt);%w_{32}
\draw [fill=uuuuuu] (2.5264865420783305,7.278389194144292) circle (2.0pt);%w_{31}
\draw [fill=uuuuuu] (-1.3038295043376478,-5.888322215410651) circle (2.0pt);%w_{21}
\draw [fill=uuuuuu] (3.179723019591773,-10.87743852542114) circle (2.0pt);%w_{23}
\draw [fill=uuuuuu] (16.849583437660886,-10.572987291165703) circle (2.0pt);%w_{13}
\draw [fill=uuuuuu] (18.256021058308615,-5.738357970189127) circle (2.0pt);%w_{12}

\draw [fill=uuuuuu] (1.7753065570277258,-2.626096507717178) circle (2.0pt);%base point of h_1
\draw [fill=uuuuuu] (10.18018905245636,1.1951834462874285) circle (2.0pt);%base point of h_2
\draw [fill=uuuuuu] (5.002362613469258,-7.574371621169964) circle (2.0pt);%base point of h_3
\draw [fill=uuuuuu] (4.912481939192984,-3.538729346165252) circle (2.0pt);%intersection of hights

\draw [fill=green] (0.6769335190250809,0.9205506773987393) circle (2.5pt);%q_1 previous color was xdxdff
\draw [fill=green] (12.97621239524618,0.13680003766967452) circle (2.5pt);%q_2
\draw [fill=green] (10.067657827557744,-10.724032404976462) circle (2.5pt);%q_3

\draw [fill=orange] (0.5240175483120926,0.3949020280728398) circle (2.5pt);%p_1
\draw [fill=orange] (2.487939511999994,-0.17642072500000347) circle (2.5pt);%foot of p_1
\draw [fill=orange] (12.25830754861328,0.9356555791751305) circle (2.5pt);%p_2
\draw [fill=orange] (11.237679290999997,0.01845098499999805) circle (2.5pt);%foot of p_2
\draw [fill=orange] (9.127740996788008,-10.744965964681803) circle (2.5pt);%p_3
\draw [fill=orange] (9.06,-7.48) circle (2.5pt);%foot of p_3

\draw [line width=1.pt,dotted] (11.237679290999997,0.01845098499999805)
-- (12.25830754861328,0.9356555791751305);

\draw [line width=1.pt,dotted] (9.06,-7.48)
-- (9.127740996788008,-10.744965964681803);

\draw [line width=1.pt,dotted] (2.487939511999994,-0.17642072500000347)
-- (0.5240175483120926,0.3949020280728398);

\draw [line width=2.pt,dashed,color=orange] (17.804855917999998,-7.289238140000015)-- (12.25830754861328,0.9356555791751305);

\draw [line width=2.pt,dashed,color=orange] (12.25830754861328,0.9356555791751305)-- (4.6705026639999945,7.326140110000009);

\draw [line width=2.pt,dashed,color=orange] (4.6705026639999945,7.326140110000009)-- (0.5240175483120926,0.3949020280728398);

\draw [line width=2.pt,dashed,color=orange] (0.5240175483120926,0.3949020280728398)--
(0.31,-7.68);

\draw [line width=2.pt,dashed,color=orange] (0.31,-7.68)--
(9.127740996788008,-10.744965964681803);

\draw [line width=2.pt,dashed,color=orange] (9.127740996788008,-10.744965964681803)--
(17.8,-7.29);
\draw [fill=zzttqq] (4.6705026639999945,7.326140110000009) circle (3.5pt);%v_3
\draw [fill=blue] (5.574847074, 7.346281411) circle (3.0pt);%

\draw [fill=blue] (9.727701942, 3.751611375) circle (3.0pt);%

\draw [line width=2.pt,dotted,color=blue] (5.574847074, 7.346281411)-- (9.727701942, 3.751611375);

\draw [line width=2.pt,dotted,color=blue](9.727701942, 3.751611375)--(15.61611673, -2.800778966);

\draw [fill=blue] (15.61611673, -2.800778966) circle (3.0pt);%

\draw [fill=blue] (18.03043849,-6.513798055) circle (3.0pt);%
\draw [fill=zzttqq] (17.804855917999998,-7.289238140000015) circle (3.5pt);%v_1

\draw [line width=2.pt,dotted,color=blue] (15.61611673, -2.800778966)-- (18.03043849,-6.513798055);

\draw [line width=2.pt,dotted,color=blue] (18.03043849,-6.513798055) -- (17.32721968, -8.931112715) ;

\draw [fill=blue] (17.32721968, -8.931112715) circle (3.0pt);%

\draw [fill=blue] (13.45862064, -10.64850984) circle (3.0pt);%

\draw [line width=2.pt,dotted,color=blue] (17.32721968, -8.931112715) -- (13.45862064, -10.64850984) ;
\draw [line width=2.pt,dotted,color=blue] (6.623690425, -10.80073546) -- (13.45862064, -10.64850984) ;

\draw [fill=blue] (6.623690425, -10.80073546) circle (3.0pt);%

\draw [fill=blue] (1.742549690, -9.278210045) circle (3.0pt);%
\draw [fill=zzttqq] (0.30537635999999335,-7.6789815600000155) circle (3.5pt);%v_2

\draw [line width=2.pt,dotted,color=blue] (6.623690425, -10.80073546) -- (1.742549690, -9.278210045) ;
\draw [line width=2.pt,dotted,color=blue] (0.30537635999999335,-7.6789815600000155) -- (1.742549690, -9.278210045) ;

\draw [fill=blue] (-0.4992265720,-6.783651888 ) circle (3.0pt);%

\draw [fill=blue] (-0.3134479925, -2.483885769 ) circle (3.0pt);%

\draw [line width=2.pt,dotted,color=blue] (-0.3134479925, -2.483885769 ) -- (-0.4992265720,-6.783651888 );
\draw [line width=2.pt,dotted,color=blue] (-0.3134479925, -2.483885769 ) -- (1.601710030, 4.099469936 ) ;

\draw [fill=blue] (1.601710030, 4.099469936 ) circle (3.0pt);%

\draw [fill=blue] (3.598494603, 7.302264652) circle (3.0pt);%

\draw [line width=2.pt,dotted,color=blue]   (3.598494603, 7.302264652)-- (1.601710030, 4.099469936 ) ;

\draw [line width=2.pt,dotted,color=blue]   (3.598494603, 7.302264652)--  (5.574847074, 7.346281411) ;

\end{scriptsize}
\end{tikzpicture}}
\caption{Illustration
for  \begin{color}{blue} $D$ \end{color} $ =
\frac{1}{2}$\begin{color}{green}$H_1$ \end{color}$+\,\frac{1}{2}H_2$.
}
\label{Fig3}
\end{figure}

\end{center}

\bigskip

 Six of these sides are subsets of the six sides of $H_2$ as $h_{H_1}(u)=h_{H_2}(u)$
holds for any $u\in\cU(H_2)$, and the other six sides of $D$
are parallel to the sides of $H_1$. We prove
\begin{equation}
\label{KinD}
K\subset D
\end{equation}
indirectly, so we assume that there exists
$x\in  K\setminus D$, and seek a contradiction. Then there exists $u\in\cU(D)$ such that
$h_K(u)\ge \langle x,u\rangle >h_D(u)$. Since $x\in K\subset H_2$, $u$ is normal to a side of $D$ parallel to a side of $H_1$. Let $u$ be normal to $[v_i,q_j]$, $i\neq j$, so that $w_{ij}$ is the vertex of $H_2$ where $u$ is a normal to $H_2$.
For
$$
\Delta=h_{H_2}(u)-h_{H_1}(u)=\langle u,w_{ij}-q_j\rangle,
$$
we have
$$
\Delta\cdot\|q_j-v_i\|=2A([q_j,v_i,w_{ij}])=t_jh_j\|w_{ij}-q_j\|.
$$
We have $h_j\geq 1.68$ by \eqref{bhnew}, and since $t_i\leq   \frac1{12}$ by \eqref{titau}
and $h_i<2$,
we obtain using (\ref{piqi}) that
$$\|q_j-v_i\|\leq \|p_i-v_i\|+0.1\le \frac{a_i}{2}+\frac{1}{12}h_i+0.1\le \frac{3}{2}+\frac{1}{6}+0.1<  2,$$
 and
$$\|w_{ij}-q_j\|\geq \|w_{ij}-p_j\|-0.1\ge \frac{a_j}{2}-0.1\ge 1-0.1=0.9.$$
Thus we get
$$
\Delta\geq \frac{1}{2}\, {t_j}\cdot 1.68\cdot 0.9\geq 60\sqrt{\varepsilon} .
$$
We deduce that
$$
\langle u,x-q_j\rangle>h_{D}(u)-h_{H_1}(u)=\frac12(h_{H_2}(u)-h_{H_1}(u))=\frac{\Delta}2\geq
30 \sqrt{\varepsilon} ,
$$
therefore
$$
A([q_j,v_i,x])=\frac{\langle u,x-q_j\rangle\cdot \|q_j-v_i\|}2\geq 15\sqrt{\varepsilon}
\cdot \|q_j-v_i\|.
$$
Note that $15\sqrt{\varepsilon}  \le 0.02$ and $\|q_j-v_i\|\ge\|p_j-v_i\|-0.1\ge  0.9$.
It follows from Claim~\ref{triangleineq} and $\|q_j-v_i\|\leq 2$ that
$$
\|x-v_i\|+\|x-q_j\|\geq \|q_j-v_i\|+\frac1{\|q_j-v_i\|}\cdot\left(15\sqrt{\varepsilon}\right)^2
\geq \|q_j-v_i\|+15\varepsilon.
$$
Writing $\widetilde{H}$ to denote the polygon that is the convex hull of $x$ and $H_1$, we have
that $[v_i,x]$ and $[x,q_j]$ are sides of $\widetilde{H}$, and hence
$$
L(K)\geq L(\widetilde{H})\geq L(H_1)+15\varepsilon,
$$
contradicting (\ref{LKH1}). In turn, we conclude (\ref{KinD}).

\medskip

\noindent {\bf Step 6} {\it $A(H_2,-H_2)-A(K,-K)$ is too large, assuming (\ref{tibig0}).}

\medskip

In order to calculate $A(D,-D)$ for
$D=\frac12\,H_1+\frac12\, H_2$,
  we claim that $D_0=\frac12\,T+\frac12\, H_2$ satisfies
\begin{equation}
\label{D0D}
A(D,-D)=A(D_0,-D_0)+A(T)(t_1t_2+t_2t_3+t_1t_3).
\end{equation}

\begin{center}
\begin{figure}

\centering
\scalebox{0.4}{
\definecolor{xdxdff}{rgb}{0.49019607843137253,0.49019607843137253,1.}
\definecolor{uuuuuu}{rgb}{0.26666666666666666,0.26666666666666666,0.26666666666666666}
\definecolor{zzttqq}{rgb}{0.6,0.2,0.}
\definecolor{ududff}{rgb}{0.30196078431372547,0.30196078431372547,1.}
\definecolor{qqwuqq}{rgb}{0.,0.39215686274509803,0.}
\begin{tikzpicture}[line cap=round,line join=round,>=triangle 45,x=1.0cm,y=1.0cm]
%\clip(2.,-10.) rectangle (10.,12.);
\fill[line width=2.pt,color=zzttqq,fill=zzttqq,fill opacity=0.10000000149011612] (4.6705026639999945,7.326140110000009) -- (17.804855917999998,-7.289238140000015) -- (0.30537635999999335,-7.6789815600000155) -- cycle;
\draw [line width=2.pt,color=zzttqq] (4.6705026639999945,7.326140110000009)-- (17.804855917999998,-7.289238140000015);
\draw [line width=2.pt,color=zzttqq] (17.804855917999998,-7.289238140000015)-- (0.30537635999999335,-7.6789815600000155);
\draw [line width=2.pt,color=zzttqq] (0.30537635999999335,-7.6789815600000155)-- (4.6705026639999945,7.326140110000009);
\draw [line width=2.pt] (6.479191485415802,7.3664227118132795)-- (18.256021058308615,-5.738357970189127);
\draw [line width=2.pt] (2.5264865420783305,7.278389194144292)-- (-1.3038295043376478,-5.888322215410651);
\draw [line width=2.pt] (3.179723019591773,-10.87743852542114)-- (16.849583437660886,-10.572987291165703);
\draw [line width=2.pt,dash pattern=on 3pt off 3pt,color=green] (4.6705026639999945,7.326140110000009)-- (12.97621239524618,0.13680003766967452);
\draw [line width=2.pt,dash pattern=on 3pt off 3pt,color=green] (12.97621239524618,0.13680003766967452)-- (17.804855917999998,-7.289238140000015);
\draw [line width=2.pt,dash pattern=on 3pt off 3pt,color=green] (17.804855917999998,-7.289238140000015)-- (10.067657827557744,-10.724032404976462);
\draw [line width=2.pt,dash pattern=on 3pt off 3pt,color=green] (10.067657827557744,-10.724032404976462)-- (0.30537635999999335,-7.6789815600000155);
\draw [line width=2.pt,dash pattern=on 3pt off 3pt,color=green] (0.30537635999999335,-7.6789815600000155)-- (0.6769335190250809,0.9205506773987393);
\draw [line width=2.pt,dash pattern=on 3pt off 3pt,color=green] (0.6769335190250809,0.9205506773987393)-- (4.6705026639999945,7.326140110000009);
\draw [line width=1.pt,dotted] (1.7753065570277258,-2.626096507717178)-- (17.804855917999998,-7.289238140000015);
\draw [line width=1.pt,dotted] (5.002362613469259,-7.5743716211699645)-- (4.6705026639999945,7.326140110000009);
\draw [line width=1.pt,dotted] (10.18018905245636,1.1951834462874285)-- (0.30537635999999335,-7.6789815600000155);

\draw [line width=2.pt,] (2.5264865420783305,7.278389194144292)-- (6.479191485415802,7.3664227118132795);
\draw [line width=2.pt] (18.256021058308615,-5.738357970189127)-- (16.849583437660886,-10.572987291165703);
\draw [line width=2.pt] (-1.3038295043376478,-5.888322215410651)-- (3.179723019591773,-10.87743852542114);

\draw (2.916657273999994,1.3630657839999971) node[anchor=north west] {\scalebox{1.9}{$a_1$}};
\draw (11.374089487999996,-1.0533434200000074) node[anchor=north west] {\scalebox{1.9}{$a_2$}};
\draw (9.601782915999995,-6.4) node[anchor=north west] {\scalebox{1.9}{$a_3$}};
\draw (-0.23,1.5) node[anchor=north west] {\scalebox{1.9}{\begin{color}{green}$q_1$\end{color}}};

\draw (13.25,0.413) node[anchor=north west] {\scalebox{1.9}{\begin{color}{green}$q_2$\end{color}}};
\draw (9.659218439999995,-10.880005894000026) node[anchor=north west] {\scalebox{1.9}{\begin{color}{green}$q_3$\end{color}}};

\draw (-0.4,-7.7) node[anchor=north west] {\scalebox{1.9}{\begin{color}{red}$v_2$\end{color}}};
\draw (-2.75,-5.4) node[anchor=north west] {\scalebox{1.9}{$w_{21}$}};
\draw (2.00,-11.0) node[anchor=north west] {\scalebox{1.9}{$w_{23}$}};

\draw (17.8,-7.4) node[anchor=north west] {\scalebox{1.9}{\begin{color}{red}$v_1$\end{color}}};
\draw (16.8,-10.6) node[anchor=north west] {\scalebox{1.9}{$w_{13}$}};
\draw (18.4,-5.4) node[anchor=north west] {\scalebox{1.9}{$w_{12}$}};
%\draw (19.7,-6.8) node[anchor=north west] {\scalebox{1.9}{$w_{12}'$}};

%\draw (17.5,-8.5) node[anchor=north west] {$t_3a_1$};
%\draw (18.1,-6.3) node[anchor=north west] {$t_2a_1$};
\draw (3.95,8.4) node[anchor=north west] {\scalebox{1.9}{\begin{color}{red}$v_3$\end{color}}};
\draw (6.7,7.7) node[anchor=north west] {\scalebox{1.9}{$w_{32}$}};
%\draw (5.25,9.5) node[anchor=north west] {\scalebox{1.9}{$w_{32}'$}};
\draw (1.05,7.6) node[anchor=north west] {\scalebox{1.9}{$w_{31}$}};
%\draw (3.0,8.0) node[anchor=north west] {$t_1a_3$};
%\draw (5.0,8.0) node[anchor=north west] {$t_2a_3$};
\begin{scriptsize}
%\draw [fill=zzttqq] (4.6705026639999945,7.326140110000009) circle (3.5pt);
%\draw [fill=zzttqq] (17.804855917999998,-7.289238140000015) circle (3.5pt);
%\draw [fill=zzttqq] (0.30537635999999335,-7.6789815600000155) circle (3.5pt);
\draw [fill=uuuuuu] (6.479191485415802,7.3664227118132795) circle (2.0pt);%w_{32}
\draw [fill=uuuuuu] (2.5264865420783305,7.278389194144292) circle (2.0pt);%w_{31}
\draw [fill=uuuuuu] (-1.3038295043376478,-5.888322215410651) circle (2.0pt);%w_{21}
\draw [fill=uuuuuu] (3.179723019591773,-10.87743852542114) circle (2.0pt);%w_{23}
\draw [fill=uuuuuu] (16.849583437660886,-10.572987291165703) circle (2.0pt);%w_{13}
\draw [fill=uuuuuu] (18.256021058308615,-5.738357970189127) circle (2.0pt);%w_{12}

%\draw [fill=uuuuuu] (5.121667804308615,8.877020279810893) circle (2.0pt);%w_{32}'
%\draw [fill=uuuuuu] (19.61354476, -7.248955540) circle (2.0pt);%w_{12}'

\draw [fill=uuuuuu] (1.7753065570277258,-2.626096507717178) circle (2.0pt);%base point of h_1
\draw [fill=uuuuuu] (10.18018905245636,1.1951834462874285) circle (2.0pt);%base point of h_2
\draw [fill=uuuuuu] (5.002362613469258,-7.574371621169964) circle (2.0pt);%base point of h_3
\draw [fill=uuuuuu] (4.912481939192984,-3.538729346165252) circle (2.0pt);%intersection of hights

\draw [fill=green] (0.6769335190250809,0.9205506773987393) circle (2.5pt);%q_1 previous color was xdxdff
\draw [fill=green] (12.97621239524618,0.13680003766967452) circle (2.5pt);%q_2
\draw [fill=green] (10.067657827557744,-10.724032404976462) circle (2.5pt);%q_3

\draw [line width=1.pt,dotted] (11.237679290999997,0.01845098499999805)
-- (12.25830754861328,0.9356555791751305);

\draw [line width=1.pt,dotted] (9.06,-7.48)
-- (9.127740996788008,-10.744965964681803);

\draw [line width=1.pt,dotted] (2.487939511999994,-0.17642072500000347)
-- (0.5240175483120926,0.3949020280728398);

\draw [fill=zzttqq] (4.6705026639999945,7.326140110000009) circle (3.5pt);%v_3

%\draw [fill=uuuuuu] (6.479191485415802,7.3664227118132795) circle (2.0pt);%w_{32}
\draw [fill=blue] (9.727701942, 3.751611375) circle (3.0pt);%
%\draw [fill=green] (12.97621239524618,0.13680003766967452) circle (2.5pt);%q_2

\draw [line width=2.pt,dotted,color=blue] (5.574847074, 7.346281411)-- (9.727701942, 3.751611375);

\draw [line width=2.pt,dotted,color=blue](9.727701942, 3.751611375)--(15.61611673, -2.800778966);

\draw [fill=blue] (15.61611673, -2.800778966) circle (3.0pt);%
%\draw [fill=uuuuuu] (18.256021058308615,-5.738357970189127) circle (2.0pt);%w_{12}

\draw [fill=zzttqq] (17.804855917999998,-7.289238140000015) circle (3.5pt);%v_1

\draw [line width=2.pt,dotted,color=blue] (15.61611673, -2.800778966)-- (18.03043849,-6.513798055);

\draw [line width=2.pt,dotted,color=blue] (18.03043849,-6.513798055) -- (17.32721968, -8.931112715) ;

%\draw [fill=uuuuuu] (16.849583437660886,-10.572987291165703) circle (2.0pt);%w_{13}
\draw [fill=blue] (13.45862064, -10.64850984) circle (3.0pt);%
%\draw [fill=green] (10.067657827557744,-10.724032404976462) circle (2.5pt);%q_3

\draw [line width=2.pt,dotted,color=blue] (17.32721968, -8.931112715) -- (13.45862064, -10.64850984) ;
\draw [line width=2.pt,dotted,color=blue] (6.623690425, -10.80073546) -- (13.45862064, -10.64850984) ;

\draw [fill=blue] (6.623690425, -10.80073546) circle (3.0pt);%
%\draw [fill=uuuuuu] (3.179723019591773,-10.87743852542114) circle (2.0pt);%w_{23}

%\draw [fill=blue] (1.742549690, -9.278210045) circle (3.0pt);%v_2w_23

\draw [fill=blue] (18.03043849,-6.513798055) circle (3.0pt);%w_12v_1
\draw [fill=blue] (5.574847074, 7.346281411) circle (3.0pt);%v_3w_32
\draw [line width=2.pt,dotted,color=magenta] (18.03043849,-6.513798055) --
(5.574847074, 7.346281411);

\draw [fill=blue] (17.32721968, -8.931112715) circle (3.0pt);%w_13v_1
\draw [fill=blue] (1.742549690, -9.278210045) circle (3.0pt);%v_2w_23
\draw [line width=2.pt,dotted,color=magenta] (17.32721968, -8.931112715) -- (1.742549690, -9.278210045) ;

\draw [fill=blue] (-0.4992265720,-6.783651888 ) circle (3.0pt);%v_2w_21
\draw [fill=blue] (3.598494603, 7.302264652) circle (3.0pt);%w_31v_3
\draw [line width=2.pt,dotted,color=magenta] (-0.4992265720,-6.783651888 ) -- (3.598494603, 7.302264652) ;

\draw [fill=zzttqq] (0.30537635999999335,-7.6789815600000155) circle (3.5pt);%v_2

\draw [line width=2.pt,dotted,color=blue] (6.623690425, -10.80073546) -- (1.742549690, -9.278210045) ;
\draw [line width=2.pt,dotted,color=blue] (0.30537635999999335,-7.6789815600000155) -- (1.742549690, -9.278210045) ;

%\draw [fill=uuuuuu] (-1.3038295043376478,-5.888322215410651) circle (2.0pt);%w_{21}
\draw [fill=blue] (-0.3134479925, -2.483885769 ) circle (3.0pt);%
%\draw [fill=green] (0.6769335190250809,0.9205506773987393) circle (2.5pt);%q_1

\draw [line width=2.pt,dotted,color=blue] (-0.3134479925, -2.483885769 ) -- (-0.4992265720,-6.783651888 );
\draw [line width=2.pt,dotted,color=blue] (-0.3134479925, -2.483885769 ) -- (1.601710030, 4.099469936 ) ;

\draw [fill=blue] (1.601710030, 4.099469936 ) circle (3.0pt);%
%\draw [fill=uuuuuu] (2.5264865420783305,7.278389194144292) circle (2.0pt);%w_{31}

\draw [line width=2.pt,dotted,color=blue]   (3.598494603, 7.302264652)-- (1.601710030, 4.099469936 ) ;

\draw [line width=2.pt,dotted,color=blue]   (3.598494603, 7.302264652)--  (5.574847074, 7.346281411) ;

\draw [line width=2.pt,dotted,color=magenta] (3.598494603, 7.302264652) --
(5.574847074, 7.346281411);

\draw [line width=2.pt,dotted,color=magenta] (17.32721968, -8.931112715) --
(18.03043849,-6.513798055);

\draw [line width=2.pt,dotted,color=magenta] (1.742549690, -9.278210045) --
 (-0.4992265720,-6.783651888 );

\end{scriptsize}
\end{tikzpicture}}
\caption{Illustration
for  \begin{color}{blue} $D$ \end{color} $ =
 \frac{1}{2}$\begin{color}{green}$H_1$ \end{color}$+\,\frac{1}{2}H_2$ and
 \begin{color}{magenta}$D_0$\end{color}$=\frac{1}{2}$\begin{color}{brown}$T$\end{color}$+\frac{1}{2} H_2$.
}
\label{Fig3new}
\end{figure}
\end{center}

We prove  (\ref{D0D}) by applying Betke's formula (\ref{BetkeP-P}) three times.
Let $b_i$ be the side of $D$ containing $q_i$ (and hence $b_i$ is contained in the ``long" side of $H_2$ parallel to
$a_i$), and let $\tilde{u}_i\in\cU(D)$ be the normal to $D$ at the vertex $v_i$ of $T$, $i=1,2,3$,
and hence $-\tilde{u}_i$ is the exterior unit normal to $b_i$.  In addition,
let $d_i$ be the diagonal of $D$ that cuts off $b_i$ and the two sides neighboring $b_i$ from $D$,
and hence $d_i$ is parallel to $b_i$, $i=1,2,3$. We write $\nu_{ij}$ to denote the  exterior unit normal to the side of $D$ cut off by and parallel to $[q_i,v_j]$, $i\neq j$, and hence $\nu_{ij}$ and $\nu_{ik}$ are the normals to the two sides neighboring $b_i$,
$\{i,j,k\}=\{1,2,3\}$.

Now $d_1$ dissects $D$ into a trapezoid and a polygon $D_1$ with $10$ sides, and on the way to verify (\ref{D0D}), we first claim that
\begin{equation}
\label{D1D}
A(D,-D)=A(D_1,-D_1)+S_D(\tilde{u}_1)S_D(\nu_{1,2})|\det(\tilde{u}_1,\nu_{1,2})|.
\end{equation}
To prove \eqref{D1D}, we choose a unit vector $u_1\neq \tilde{u}_1$ very close to $\tilde{u}_1$ and such that
$\langle u_1, v_2-v_3\rangle>0$.
When we apply Betke's formula (\ref{BetkeP-P}) to calculate the difference $A(D,-D)-A(D_1,-D_1)$ using $u_1$
as the reference vector, we deduce after cancellation of summands common to $A(D,-D)$ and $A(D_1,-D_1)$,
 that the exterior unit normal $-\tilde{u}_1$ to the sides $b_1$ of $D$ and $d_1$ of $D_1$ does not occur in either term and precisely one of the two exterior unit normals of the two sides of $D$ neighboring $b_1$, in this case $\nu_{12}$,  occurs.
To see this, we observe that  if $\{i,j,k\}=\{1,2,3\}$, then
\begin{equation}
\label{angle-ui}
\angle(-\tilde{u}_i,\nu_{ij})<0.16
\end{equation}
because $\angle(-\tilde{u}_i,\nu_{ij})=\angle(q_i-v_j,v_k-v_j)$ satisfies
$\tan\angle(-\tilde{u}_i,\nu_{ij})\le (t_ih_i)/(1-0.06)<0.16$
using the estimates \eqref{piqi}, $h_i<1.77$ and $t_i\leq \frac1{12}$ (cf.~\eqref{titau}).
On the one hand, if $\{i,j\}=\{2,3\}$, then \eqref{alhai}, \eqref{angle-ui}
and  $\sqrt{\varepsilon}<(6\cdot 180)^{-1}$ yield
$$
\angle(\tilde{u}_1,\nu_{i1})=\angle(\tilde{u}_1,-\tilde{u}_i)-\angle(-\tilde{u}_i,\nu_{i1})=
\alpha_j-\angle(-\tilde{u}_i,\nu_{i1})\geq
\frac{\pi}3-10\sqrt{\varepsilon}-0.16>0.87;
$$
 therefore, the angle of $\tilde{u}_1$ with any other exterior unit normal to $D$ or $D_1$ is at least $0.87$.
On the other hand,  $\angle(-\tilde{u}_1,\nu_{1j})<0.16$ for $j=2,3$ by \eqref{angle-ui}, concluding the proof
of \eqref{D1D}.

Since
$$
S_D(\nu_{1,2})=\frac{1}{2}\|q_1-v_{2}\|,\quad S_D(\tilde{u}_1)=\frac{1}{2}(t_2+t_3)a_1,\quad
|\det(\tilde{u}_1,\nu_{12})|=\frac{t_1h_1}{\|q_1-v_{2}\|},
$$
we deduce from \eqref{D1D} that
\begin{equation}
\label{D1D-2}
A(D,-D)=A(D_1,-D_1)+\frac12\,t_1(t_2+t_3)\,A(T).
\end{equation}

Next we observe that $d_2$ dissects $D_1$ into a trapezoid and a polygon $D_2$ with $8$ sides.
Choosing a unit vector $u_2\neq \tilde{u}_2$ very close to $\tilde{u}_2$,
and applying Betke's formula (\ref{BetkeP-P}) to calculate $A(D_1,-D_1)-A(D_2,-D_2)$ with $u_2$ as reference vector,
we conclude as before that
$$
A(D_1,-D_1)=A(D_2,-D_2) +\frac{1}{2}t_2(t_1+t_3)A(T).
$$
Finally, the analogous argument for a
 a unit vector $u_3\neq \tilde{u}_3$ very close to $\tilde{u}_3$ implies that
$$
A(D_2,-D_2)=A(D_0,-D_0)+ \frac{1}{2}t_3(t_1+t_2)A(T).
$$
Thus we arrive at
$$
A(D,-D)=A(D_0,-D_0)+ \frac{1}{2}[t_1(t_2+t_3)+t_2(t_1+t_3)+   t_3(t_1+t_2)]A(T),
$$
which completes the proof of (\ref{D0D}).

\medskip

We recall that $A(T,-T)=2A(T)$, $A(H_2,-H_2)=2A(T)(1+t_1t_2+t_2t_3+t_3t_1)$, and
 observe that $A(H_2,-T)=A(T,-H_2)=A(T,-T)$ by the symmetry and rigid motion invariance of the mixed area  and Minkowski's formula (\ref{MinkowskiPQ}).
Thus (\ref{D0D}) and the linearity of the mixed area imply
\begin{align*}
A(D,-D)&= A(D_0,-D_0)+A(T)(t_1t_2+t_2t_3+t_1t_3)\\
&=A\left(\frac12\,T+\frac12\, H_2,-\frac12\,T-\frac12\, H_2\right)+A(T)(t_1t_2+t_2t_3+t_1t_3)\\
&=
\frac14\cdot 2A(T)+\frac24\cdot 2A(T)+\frac14\cdot 2A(T)(1+t_1t_2+t_2t_3+t_3t_1)\\
&\qquad\qquad +A(T)(t_1t_2+t_2t_3+t_1t_3)\\
&= 2A(T)+\frac{3}{2}A(T)(t_1t_2+t_2t_3+t_3t_1).
\end{align*}

We deduce from (\ref{AKH2}), $K\subset D$ ({\it cf.} (\ref{KinD})), (\ref{tibig}) and $A(T)\ge 3/2$ that
\begin{align*}
5\varepsilon&\geq A(H_2,-H_2)- A(K,-K)\geq  A(H_2,-H_2)- A(D,-D)\\
&=2A(T)(1+t_1t_2+t_2t_3+t_3t_1)-2A(T)-\frac{3}{2}A(T)(t_1t_2+t_2t_3+t_3t_1)\\
&=  \frac{1}{2}A(T)(t_1t_2+t_2t_3+t_3t_1)\geq \frac{1}{2} \cdot
\frac{3}{2}\cdot 2\cdot (86\cdot \gamma)\sqrt{\varepsilon}^{\,2}\ge \gamma^2\varepsilon,
\end{align*}
which is
a contradiction, proving that \eqref{tibig0} does not hold.
Therefore the argument in Step~3 proves Proposition~\ref{Betke-Weil-local-stab}.
\end{proof}

\section{Proof of Theorem~\ref{Betke-Weil-stab}}
\label{secglobal}

 Before starting the actual proof of Theorem~\ref{Betke-Weil-stab}, we recall Proposition~\ref{Betke-Weil-deform} and Lemma~\ref{Betke-Weil-regular}
proved in essence by Betke and Weil \cite{BeW91} as Lemma~1 and Lemma~2 in \cite{BeW91}.
We slightly modified the argument from of \eqref{BetkeP-P} and added an observation which turns out
to be useful.

\begin{prop}[Betke, Weil (1991)]
\label{Betke-Weil-deform}
If $P$ is a polygon with $k\ge 3$ sides, and $P$ is not a regular polygon with an odd number of sides, then
there exists  a polygon $P'$ with $k$ sides and arbitrarily close to $P$ such that
$$
\frac{L(P')^2}{A(P',-P')}<\frac{L(P)^2}{A(P,-P)}.
$$
\end{prop}
\begin{proof}
For a vertex $v$ of the polygon $P$, we write $N_P(v)$ for the normal cone of $P$ at $v$; namely,
if $u_1,u_2\in \cU(P)$ are the exterior unit normals of the two sides meeting at $v$, then
$N_P(v)={\rm pos}\{u_1,u_2\}$.

\medskip

\noindent{\bf Case 1} {\it There exist vertices $v_1$ and $v_2$ of $P$
such that $-N_P(v_1)\subset N_P(v_2)$.}

In this case, $P$ is not a triangle. Let $v_3,v_4$ be the neighbors of $v_1$. For a
$v'_1\in[v_1,v_4]\backslash\{v_1,v_4\}$, let $P'$ be obtained from $P$ by replacing the vertex $v_1$ by $v'_1$.
In particular, there exists a unit vector $w\in -{\rm int}\,N_P(v_1)$ orthogonal to $[v'_1,v_3]$.
 Using this vector
$w$ in \eqref{BetkeP-P} and the property $-N_P(v_1)\subset N_P(v_2)$, we get $A(P',-P')=A(P,-P)$, while obviously
$L(P')<L(P)$ by strict containment.

\medskip

\noindent{\bf Case 2} {\it There exist no vertices $v_1$ and $v_2$ of $P$
such that $-N_P(v_1)\subset N_P(v_2)$ as in Case~1, but there exist a vertex $v$ and a side $e$
with exterior normal $u_0\in \cU(P)$
such that $-u_0\in N_P(v)$ and $-u_0$ does not halve the angle of $N_P(v)$. }

In this case, $P$ does not have any parallel sides, since we are not in Case~1.
 The line $\ell$ through $v$ parallel to $e$ is a support line of $P$.
 Let $v_4$ denote the point preceding $v$ and $v_3$ the point following $v$ on $\partial P$ in the clockwise order. Let $\nu_0$ denote the unit vector orthogonal to $u_0$ such that $(v_4-v)/\|v_4-v\|,(v_3-v)/\|v_3-v\|,\nu_0$ are in counter-clockwise order on the unit circle. Let $\alpha_4$ denote the angle enclosed by $v_4-v$ and $-\nu_0$ and $\alpha_3$ denote the angle enclosed by $v_3-v$ and $\nu_0$. Let $u_i$ denote the exterior unit normal of $[v_i,v]$ for $i=3,4$. Since
 $-u_0$ does not halve the angle enclosed by $u_3,u_4$, we have $\alpha_3\neq\alpha_4$. We may assume that
 $\alpha_4>\alpha_3$. By Fermat's principle, moving $v$ along $\ell$ an arbitrarily small amount in the direction of $\nu_0$ to $v'$ and denoting by $P'$ the polygon obtained from $P$ by replacing $v$ by $v'$, we get $L(P')<L(P)$. Clearly, we thus still get a convex $k$-gon if $v'$ is sufficiently close to $v$.

To prove $A(P,-P)=A(P',-P')$, we denote by $Q$ the (nonempty) convex hull of the (common) vertices
of $P$ and $P'$ (thus $v,v'$ are removed). Note that in the case where $P$ (and hence also $P'$) is a triangle, $Q$ is a segment. We consider the triangles  $\triangle =[v_3,v_4,v]$ and $\triangle' =[v_3,v_4,v']$, hence we have $P=Q\cup \triangle$ and $P'=Q\cup \triangle'$. For the segment $I=[v_3,v_4]$ we have
$
A(\triangle,-I)=A(\triangle , I)=A(\triangle)$.
Using this and the additivity of the mixed area in both arguments, we obtain
$$
A(P,-P)=A(Q,-Q)+2A(Q,-\triangle)-2A(Q,-I).
$$
A similar expression is obtained for $P'$ with $\triangle$ replaced by $\triangle'$.

We choose $v_5,v_6$ such that $e=[v_5,v_6]$ and $v_6-v_5$ is a positive multiple of $v'-v$.
If $v'$ is sufficiently close to $v$, then the first assumption in Claim 2 ensures that for the
exterior unit normals $u_j$ of the sides of $Q$ between $v_4$ and $v_5$ (in counter-clockwise order),
we have
$$h_\triangle (u_j)=\langle v_5,u_j\rangle=h_{\triangle'}(u_j)$$
and
that for the
exterior unit normals $u_j$ of the sides of $Q$ between $v_6$ and $v_3$ (in counter-clockwise order),
we have
$$h_\triangle (u_j)=\langle v_6,u_j\rangle=h_{\triangle'}(u_j).$$
But then (writing $e$ also for the length of the edge $e$)
$$
A(Q,-\triangle)=A(Q,-\triangle')=e\langle u_0,-v\rangle-e\langle u_0,-v'\rangle=e\langle u_0,v'-v\rangle=0,
$$
which proves the assertion.

\medskip

\noindent{\bf Case 3} {\it There exist no vertices $v_1$ and $v_2$ of $P$
such that $-N_P(v_1)\subset N_P(v_2)$, and for any vertex $v$ and side $e$
with exterior normal $u_0\in \cU(P)$
such that $-u_0\in N_P(v)$ the vector $u_0$ halves the angle of $N_P(v)$,
but $P$ is not a regular polygon with an odd number of sides.}

Let $u_1,\ldots,u_k$ be the exterior unit normals to the sides of $P$ in clockwise order. We observe that no closed half plane having the origin on its boundary contains $u_1,\ldots,u_k$.
We set $u_{i+pk}=u_i$ for any $i=1,\ldots,k$ and $p\in\Z$.

Since there exist no vertices $v_1$ and $v_2$ of $P$ such that $-N_P(v_1)\subset N_P(v_2)$, we have
$u_i\neq -u_j$ for any $i,j$, and there exists an $m\in\{2,\ldots,k-1\}$ such that
$-u_i\in{\rm int}\,{\rm pos}\{u_{i+m-1},u_{i+m}\}$ for any $i$. The conditions in Case~3 imply that
$-u_i$ actually halves the angle $\angle(u_{i+m-1},u_{i+m})$ between $u_{i+m-1}$ and $u_{i+m}$.

Let us assume that $\angle(u_{m},u_{m+1})$ is minimal among the angles of the from $\angle(u_j,u_{j+1})$.
We deduce that no $u_j$ lies in ${\rm int}\,{\rm pos}\{u_1,-u_{m}\}\subset {\rm pos}\{-u_m,-u_{m+1}\}$;
therefore, $u_1=u_{2m}$. We conclude that $k=2m-1$, and hence $k$ is odd.

As $-u_i$ halves $\angle(u_{i+m-1},u_{i+m})$ for any $i$, and
$-u_{i+m-1}$ halves $\angle(u_{i-1},u_{i})$, we deduce that
$\angle(u_{i+m-1},u_{i+m})=\angle(u_{i-1},u_{i})$ for any $i$. In turn, we conclude that
all exterior angles of $P$ are $2\pi/k$.

For any fixed $i\in\Z$, we write $e_i$ to denote both the side corresponding to $u_i$ of $P$ and its length, and
 $\ell_i$ to denote the line containing $e_i$. Since $P$ is not regular, the side lengths of $P$ are not the same;
therefore, there exists $1\leq p<q\leq k$ such that
\begin{equation}
\label{epqdiff}
e_{p-1}+e_p\neq e_{q-1}+e_q.
\end{equation}

Fixing $i\in\{1,\ldots,k\}$, if $|t|$ is small, then let $P_{i,t}$ be the $k$-gon
bounded by the lines $\ell_j$, $j\in\{1,\ldots,k\}\backslash\{i\}$ and $\ell_i+tu_i$.
It follows  that
\begin{align*}
L(P_{i,t})&=L(P)+2t\cdot\left(\frac1{\sin\frac{2\pi}k}-\frac1{\tan\frac{2\pi}k}\right)=L(P)+\kappa\cdot t\\
\intertext{and, choosing $w=-u_i$ in \eqref{BetkeP-P}, we see that}
A(P_{i,t},-P_{i,t})&=
A(P,-P)+(e_{i+m}+e_{i+m-1})\cdot\frac{t}{\sin\frac{2\pi}k}\cdot \sin\frac{\pi}k\\
&=
A(P,-P)+(e_{i+m}+e_{i+m-1})\cdot \varrho\cdot t.
\end{align*}
where $\kappa=2\frac{1-\cos\frac{2\pi}k}{\sin\frac{2\pi}k}$ and
$\varrho=\frac{\sin\frac{\pi}k}{\sin\frac{2\pi}k}$.
We deduce that
\begin{equation}
\label{ALder}
\left.\frac{d}{dt}\frac{A(P_{i,t},-P_{i,t})}{L(P_{i,t})^2}\right|_{t=0}
=\frac{(e_{i+m}+e_{i+m-1})\cdot \varrho \cdot L(P)-A(P,-P)\cdot 2\gamma}{L(P)^3}.
\end{equation}

We conclude from \eqref{epqdiff} and \eqref{ALder} the existence of some $i\in\{1,\ldots,k\}$ such that
$$
\left.\frac{d}{dt}\frac{A(P_{i,t},-P_{i,t})}{L(P_{i,t})^2}\right|_{t=0}
\neq 0.
$$
In particular, we can choose a $t\neq 0$ with arbitrarily small  absolute value such that
$$
\frac{A(P_{i,t},-P_{i,t})}{L(P_{i,t})^2}>\frac{A(P,-P)}{L(P)^2},
$$
and hence we can choose $P'=P_{i,t}$, completing the proof of Proposition~\ref{Betke-Weil-deform}.
\end{proof}

\bigskip

For regular polygons with an odd number of sides, we have the following estimates.

\begin{lemma}
\label{Betke-Weil-regular}
If $P$ is a regular polygon
 with an odd number $k\ge 5$ of sides, then
\begin{align}
\label{regpol-tr}
d_{\rm tr}(P)&>0.25\\[1ex]
\label{regpol-LA}
\frac{L(P)^2}{A(P,-P)}&\geq 20\sin\frac{\pi}{5}>1.1\cdot 6\sqrt{3}.
\end{align}
\end{lemma}

\begin{proof} We may assume that $L(P)=1$, and hence $A(P,-P)=(4k\sin\frac{\pi}{k})^{-1}$ according to (\ref{MinkowskiP-P}),
proving (\ref{regpol-LA}). For (\ref{regpol-tr}),  assuming that the origin is the centroid of $P$, we have
$-P\subset \left(\cos\frac{\pi}{k}\right)^{-1}P$, and hence
$A(P,-P)\leq \left(\cos\frac{\pi}{k}\right)^{-1}A(P)$. On the other hand,
if  $T_0\subset P$ is a regular triangle
with centroid $z_0$ and $P-z_0\subset (1+d)(T_0-z_0)$,
then
$$
A(P,-P)\geq A(T_0,-T_0)=2A(T_0)\geq 2(1+d)^{-2}A(P),
$$
and hence
$d_{\rm tr}(P)\geq \sqrt{2\cos\frac{\pi}{5}} -1>0.25$.
\end{proof}

\bigskip

\noindent {\it Proof of Theorem~\ref{Betke-Weil-stab}.}
Let $\varepsilon\in [0,2^{-28}]$ and let $K\subset\R^2$ be a convex domain with $L(K)^2\le \left(1+\varepsilon\right)6\sqrt{3}\,A(K,-K)$. Then $0\le \varepsilon <
(6\cdot 2400)^{-2}<(6\cdot 180)^{-2}$. We distinguish two cases.

Case 1: $d_{\rm tr}(K)\le 6^{-2}$. Then Proposition~\ref{Betke-Weil-local-stab} implies that $d_{\rm tr}(K)
\le 400\sqrt{\varepsilon}$, and the proof is finished.

Case 2: $d_{\rm tr}(K)> 6^{-2}$. We will show that in fact this case does not occur. We fix a number $\varepsilon'$ with  $0\le \varepsilon<\varepsilon'<
(6\cdot 2400)^{-2}$. Since $d_{\rm tr}(\cdot),A(\cdot,\cdot), L(\cdot)$ are continuous, there is a polygon $P\subset\R^2$ with $d_{\rm tr}(P)> 6^{-2}$ and $L(P)^2\le \left(1+\varepsilon'\right)6\sqrt{3}\,A(P,-P)$. Since
$d_{\rm tr}(\cdot),A(\cdot,\cdot), L(\cdot)$ are translation invariant, $d_{\rm tr}(\cdot)$ is scaling invariant and
$K\mapsto L^2(K)/A(K,-K)$ (for convex domains $K\subset\R^2$) is also scaling invariant, we can assume that $o\in P$ and $A(P,-P)=1$. Let $k\ge 3$ be the number of vertices of $P$. We write $\mathcal{P}_k$ for the set of all polygons $G\subset \R^2$ with at most $k$ vertices, $d_{\rm tr}(G)\ge 6^{-2}$, $o\in G$, $A(G,-G)=1$ and
 $L(G)^2\le \left(1+\varepsilon'\right)6\sqrt{3}$. Then in particular we have $P\in \mathcal{P}_k\neq\emptyset$.

 We claim that there is some $P_0\in \mathcal{P}_k$ such that $L(P_0)=\inf\{L(G):G\in\mathcal{P}_k\}$.
 For the proof, it is sufficient to consider a minimizing sequence $P_i\in\mathcal{P}_k$ with $L(P_i)\le L(P)$ for $i\in\N$. Let $B^2$ denote the unit disc with center at the origin. Then $P_i\subset L(P)B^2$ for $i\in\N$.
 An application of Blaschke's selection theorem shows that the sequence $P_i$, $i\in\N$, has a convergent
 subsequence with limit $P_0\subset\R^2$. Since all conditions involved in the definition of $\mathcal{P}_k$
 are preserved under limits and $L(\cdot)$ is continuous, we conclude that $P_0\in\mathcal{P}_k$ realizes the infimum.

 Since $d_{\rm tr}(P_0)\ge  6^{-2}>0$, $P_0$ is not a regular triangle. Assuming (for the moment) that
 $P_0$ is a regular $r$-gon with an odd number $r$,  we have  $k\ge r\ge 5$. Then Lemma~\ref{Betke-Weil-regular} shows that $L(P_0)^2\ge 1.1\cdot 6\sqrt{3}$. Since also $L(P_0)^2\le (1+\varepsilon')6\sqrt{3}$, we get $\varepsilon'\ge 0.1$, a contradiction.

 Hence $P_0\in\mathcal{P}_k$ is a $k$-gon, but not a regular polygon with an odd number of edges. Assume (for the moment) that $d_{\rm tr}(P_0)>  6^{-2}$. Proposition~\ref{Betke-Weil-deform} then shows
 that there is a $k$-gon $P_1$ such that $d_{\rm tr}(P_1)>  6^{-2}$ and
 $$
 \frac{L(P_1)^2}{A(P_1,-P_1)}<\frac{L(P_0)^2}{A(P_0,-P_0)}=L(P_0)^2.
 $$
 Again by scaling and translation invariance, we obtain a $k$-gon $P_2$ for which  $d_{\rm tr}(P_2)>  6^{-2}$,
 $A(P_2,-P_2)=1$, $o\in P_2$ and $L(P_2)^2<L(P_0)^2\le \left(1+\varepsilon'\right)6\sqrt{3}$, that is, $P_2\in\mathcal{P}_k$. But this contradicts the minimality of $L(P_0)$. Therefore we conclude that
 $d_{\rm tr}(P_0)=  6^{-2}$. Since $L(P_0)^2\le \left(1+\varepsilon'\right)6\sqrt{3}\,A(P_0,-P_0)$, it follows from
 another application of Proposition~\ref{Betke-Weil-local-stab}  that
 $$
 6^{-2}=d_{\rm tr}(P_0)\le 400\sqrt{\varepsilon'}<400\cdot(6\cdot 2400)^{-1}=6^{-2},
 $$
 a contradiction. This finally shows that the present case does not occur, which completes the argument.
\hfill $\Box$

\bigskip

\textbf{Acknowledgement.} The authors are grateful for helpful discussions with Martin Henk on the subject of the paper.

F.~Bartha was supported by the grants NKFIH KKP 129877, 2020-2.1.1-ED-2020-00003, TUDFO/47138-1/2019-ITM, EFOP-3.6.2-16-2017-0015, by the J\'anos Bolyai Research Scholarship of MTA, and by UNKP-20-5 - New national excellence program of ITM and NKFIH.

F.~Bencs is supported by the NKFIH (National Research, Development and Innovation Office, Hungary) grant KKP-133921.

K. B\"or\"oczky was supported by research grant NKFIH K 132002.

Daniel Hug was supported by research grant HU 1874/5-1 (DFG).

\end{document}